\documentclass[11pt]{amsart}
\usepackage{geometry}
\usepackage{amsmath, amssymb}
 \usepackage{amsmath,amscd}
\usepackage{amsfonts}
\usepackage{mathrsfs}
\usepackage[arrow,matrix,curve,cmtip,ps]{xy}
\usepackage{setspace}

\usepackage{pb-diagram} 
\usepackage{pb-xy}
\usepackage{graphicx}

\usepackage{amsthm}

\usepackage{color}

\usepackage{xcolor}

\allowdisplaybreaks

\newtheorem*{rep@theorem}{\rep@title}
\newcommand{\newreptheorem}[2]{
\newenvironment{rep#1}[1]{
 \def\rep@title{#2 \ref{##1}}
 \begin{rep@theorem}}
 {\end{rep@theorem}}}
\makeatother

\newreptheorem{theorem}{Theorem}
\newreptheorem{lemma}{Lemma}
\newreptheorem{proposition}{Proposition}
\newreptheorem{corollary}{Corollary}

\newtheorem{theorem}{Theorem}[section]
\newtheorem{lemma}[theorem]{Lemma}
\newtheorem{conjecture}[theorem]{Conjecture}

\newtheorem{proposition}[theorem]{Proposition}
\newtheorem{corollary}[theorem]{Corollary}

\newtheorem*{theorem*}{Theorem}
\newtheorem*{proposition*}{Proposition}
\theoremstyle{remark}
\newtheorem{remark}[theorem]{Remark}

\numberwithin{equation}{section}

\newcommand{\Z}{\mathbb{Z}}
\newcommand{\Q}{\mathbb{Q}}

\newcommand{\B}{\mathcal{B}}

\newcommand{\lnk}{\ell k}

\newcommand{\Arf}{\operatorname{Arf}}

\newcommand{\bdry}{\ensuremath{\partial}}

\newcommand{\eref}[1]{(\ref{#1})}

\newcommand{\into}{\hookrightarrow}

\newcommand{\gbar}{\overline{\delta}}
\newcommand{\abar}{\overline{\alpha}}

\newcommand{\g}{\delta}

\begin{document}
\title[Kauffman's conjectures on slice knots]{Counterexamples to Kauffman's Conjectures on Slice Knots}

\author{Tim D. Cochran$^{\dag}$}
\address{Department of Mathematics MS-136, P.O. Box 1892, Rice University, Houston, TX 77251-1892}
\email{cochran@rice.edu}

\author{Christopher William Davis}
\address{Department of Mathematics, University of Wisconsin-Eau Claire, Hibbard Humanities Hall 508,  Eau Claire WI 54702-4004}
\email{daviscw@uwec.edu}

\thanks{$^{\dag}$Partially supported by the National Science Foundation  DMS-1006908}

\date{\today}



\begin{abstract} 
In 1982 Louis Kauffman conjectured that if a knot in $S^3$ is a slice knot then on any Seifert surface for that knot there exists a homologically essential simple closed curve of self-linking zero which is itself a slice knot, or at least has Arf invariant zero. Since that time, considerable evidence has been amassed in support of this conjecture. In particular, many invariants that obstruct a knot from being a slice knot have been explictly expressed in terms of invariants of such curves on the Seifert surface. We give counterexamples to Kauffman's conjecture, that is, we exhibit (smoothly) slice knots  
that admit (unique minimal genus) Seifert surfaces on which every homologically essential simple closed curve of self-linking zero has non-zero Arf invariant and non-zero signatures.
\end{abstract}

\maketitle

\section{Introduction}

 A (classical) \textbf{knot} $K$ is an isotopy class of smooth embeddings of an oriented $S^1$ into $S^3$. $K$ is called \textbf{slice} if it bounds a  $2$-disk $D$  smoothly embedded in the 4-ball.  The disk $D$ is called a \textbf{slice disk} for $K$. These notions were first considered by Fox and Milnor in 1957 in the context of the study of singularities of surfaces in $4$-manifolds ~\cite{FoMi57,FoMi}. The question of which knots are slice knots is thus intimately related to local obstructions arising in a surgery-theoretic attempt to classify $4$-manifolds ~\cite{CF}.
 
While not every knot is a slice knot, every knot is the boundary of an embedding of a compact oriented embedded surface, $F$, in $S^3$ that is called a \textbf{Seifert surface} for $K$. In the late 1960's, Jerome Levine began a landmark program aimed towards deciding if a given knot is a slice knot, by 
 studying one of its Seifert surfaces ~\cite{L5}.
 
First, Levine found obstructions to a knot being a slice knot derived merely from the linking numbers of certain circles on $F$. Specifically he considered the \textbf{Seifert form}, a bilinear form $\beta_F:H_1(F)\times H_1(F)\to \Z$ given by $\beta_F([x],[y])= \lnk(x,y^+)$ where $x,y$ are oriented circles on $F$, $y^+$ denotes the result of pushing $y$ off of $F$ in the positive normal direction, and $\lnk$ denotes the linking number. Levine proved that if $K$ is a slice knot then \textit{any} associated Seifert form is \textbf{metabolic}, meaning that there exists  half-rank summand $\Z^g\subset H_1(F)$ on which $\beta_F$ is identically zero  ~\cite[Lemma 2]{L5}. This is equivalent to saying that if $K$ is a slice knot then, for any genus $g$ Seifert surface $F$, there exists a link $\{d_1,...,d_g\}$, of $g$ circles disjointly embedded on $F$, representing a rank $g$ summand of $H_1(F)$, for which $\lnk(d_i,d_j^+)=0$ for all $i,j$. We will call this a \textbf{set of surgery curves}, or, if $g=1$, a \textbf{surgery curve} (this link is also sometimes called a \textbf{derivative of $K$} ~\cite{derivatives}). Any knot whose Seifert form is metabolic is called an \textbf{algebraically slice knot}. Using this, Levine considered various numerical invariants of the Seifert form, certain signatures and discriminants, that obstruct a knot from being a slice knot. 

Secondly, Levine  went on to show that, for higher dimensional knots ($S^{2n-1}\hookrightarrow S^{2n+1}, n>1$),  (the analogues of) these linking numbers are the \textit{only} obstructions to $K$ being a slice knot, that is, \textit{any algebraically slice knot is a slice knot}. His method was to try to show that,  for any Seifert surface for a slice knot, there exists a set of surgery curves that is itself a slice link (recall that a \textbf{slice link} is one for which the components bound disjoint disks $D_i$ in $B^4$). For, if such disks exist,  then the Seifert surface for $K$ can be transformed (by a process called ambient surgery) to an embedded disk by replacing an annular neighborhood of each surgery curve with two copies of $D_i$, implying that $K$ is a slice knot.

In the classical dimension $n=1$ his proof fails. Indeed, in 1973 Casson and Gordon found additional obstructions to $K$ being a slice knot ~\cite{CG1,CG2}. But, significantly, their invariants also can be expressed in terms of circles on a Seifert surface (as first shown by Gilmer ~\cite{Gi3}). Whereas Levine showed that very simple invariants of surgery curves (namely their linking numbers) obstruct $K$ from being a slice knot, Casson-Gordon and Gilmer showed that certain sums of signatures of surgery curves  obstruct $K$ from being a slice knot. In recent years many more obstructions have been found and in almost every case  case they have been shown to be expressible in terms of lower order invariants of surgery curves (see below). Thus  hope has remained that Levine's philosophy/strategy was sound, namely that if $K$ is a slice knot then, for any Seifert surface, there exists a set of surgery curves that is itself a slice link.

In the simplest situation, when a (classical) slice (or merely algebraically slice) knot bounds a genus one Seifert surface, $F$,  it can be shown that there are only two isotopy classes of simple closed curves $a$ and $b$ that are essential on $F$ and satisfy $\lnk(a,a^+)=\lnk(b,b^+)=0$ ($\lnk(a,a^+)$ is called the \textbf{self linking} of $a$) ~\cite[Proposition 3]{Gi2}.  That is to say, given $F$, there are (modulo orientation) precisely two surgery curves for $K$.   An example is shown in Figure~\ref{fig:NonDoubleDerivative}.  
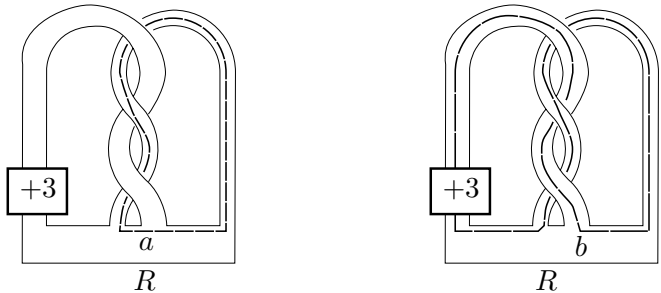
\begin{figure}[h!]
\setlength{\unitlength}{1pt}
\begin{picture}(250,100)

\put(0,5){\includegraphics[width=1.2in]{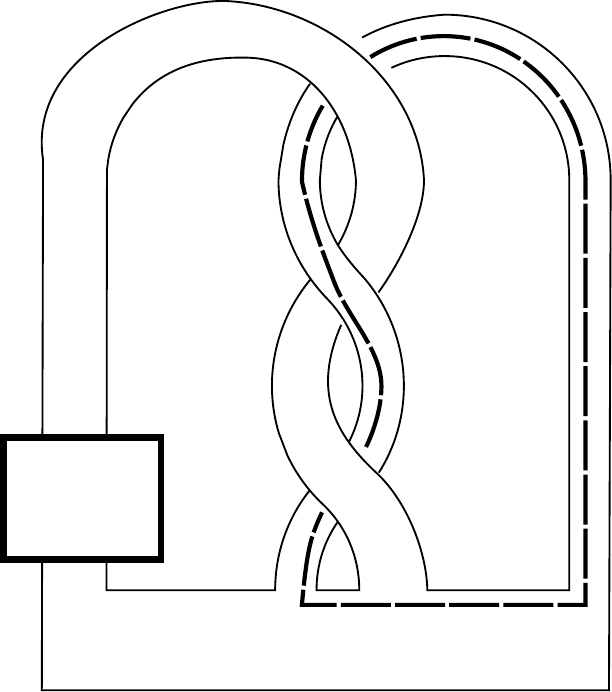}}
\put(50,10){$a$}
\put(5,30){$+3$}
\put(48,-5){$R$}

\put(160,5){\includegraphics[width=1.2in]{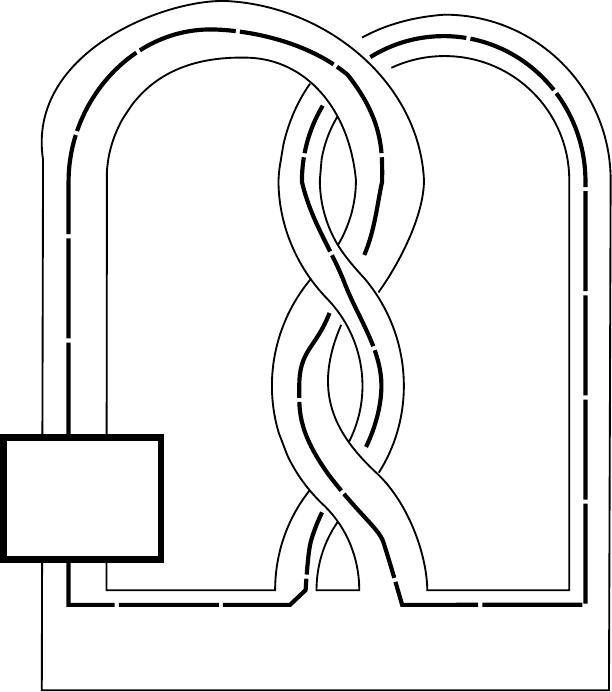}}
\put(165,30){$+3$}
\put(215,8){$b$}
\put(200,-5){$R$}
\end{picture}
\caption{Two surgery curves for the slice knot $R$.  Notice that $a$ is unknotted and $b$ is the trefoil knot.}\label{fig:NonDoubleDerivative}
\end{figure}
To reiterate, if $K$ were to bound a genus 1 Seifert surface $F$ admitting a surgery curve $a$ that is itself slice as a knot in $S^3$, then one could perform surgery on $F$ by using two copies of a slice disk for $a$ and deduce that $K$ is a slice knot.  

In 1982 Kauffman conjectured that the converse is true (in support of Levine's philosophy):
\begin{conjecture} [Kauffman's Strong Conjecture]\cite[p.226]{OnKnots}\cite[Kirby Problem N1.52]{Problems82}
If $K$ is a slice knot  and $F$ is a genus 1 Seifert surface for $K$ then there exists an essential simple closed curve $d$ on $F$ such that $\lnk(d,d^+)=0$ and $d$ is a slice knot.
\end{conjecture}

Kauffman also posed the following much weaker conjecture. Recall that the Arf invariant is a $\mathbb{Z}_2$-valued  invariant of knots that is zero for slice knots \cite{Robertello65}.  

\begin{conjecture}[Kauffman's Weak conjecture] \cite[p. 226]{OnKnots}\cite[Kirby Problem N1.52]{Problems82}\label{conj:Kauffman}
If $K$ is a slice knot  and $F$ is a genus 1 Seifert surface for $K$ then there exists an essential simple closed curve $d$ on $F$ such that $\lnk(d,d^+)=0$ and the Arf invariant of $d$ vanishes.
\end{conjecture}

In the intervening $30$ years evidence has accumulated in support of these conjectures.   As mentioned above, early work of Casson, Cooper, Gilmer, Gordon, Livingston,  Litherland and others established that if a genus one knot with nontrivial Alexander polynomial is slice (even in the topological category) then, for one of its surgery curves,  many classical algebraic invariants that obstruct sliceness (discriminants and certain sums of signatures) must vanish \cite{Gi3,Lith1984,Gi2,GL1,GL2} (for a complete survey see \cite{GiLi2011}). We call these \textit{zero$^{th}$-order invariants}. This was powerful evidence that at least one of the surgery curves must be itself at least an algebraically slice knot.  Under these same hypotheses later work of Cochran-Orr-Teichner, Cochran-Harvey-Leidy and others showed  that families of higher-order $L^2$-signatures of at least one surgery curve must also vanish ~\cite{structureInConcordance,derivatives}. Other work, using gauge theory and Heegard Floer homology, has shown that even for knots that are topologically slice there is evidence for Kauffman's conjecture ~\cite{Go1,CGompf,Hed:WHD,HedKirk2}. On the other hand, some evidence in a somewhat negative direction was recently provided by ~\cite{GiLi2011}, where  it was shown that the totality of the currently-known restrictions on the zero$^{th}$-order  invariants  of a surgery curve, $J$,  is insufficient to conclude that $J$ is algebraically slice.

In this paper we provide counterexamples to Kauffman's conjectures.

\begin{theorem}\label{counterex}
There exists a slice knot admitting a  genus 1 Seifert surface $F$ on which there is no surgery curve whose Arf invariant vanishes. Moreover there are such examples where $F$ is the unique minimal genus one Seifert surface (up to isotopy).
\end{theorem}

As mentioned above, Levine associated to a knot $K$ its signature function $\sigma_K:S^1\to \Z$.  This function (properly normalized) vanishes if $K$ is a slice knot.  This suggests a different weakening of Kauffman's Strong Conjecture:

\begin{conjecture}\label{conj:signature}
If $K$ is a slice knot  and $F$ is a genus 1 Seifert surface for $K$ then there exists a surgery curve on $F$ with vanishing Levine-Tristram signature function.
\end{conjecture}

Over the years, significant evidence has been given for this conjecture as discussed above.   

\begin{theorem}\label{SignatureCounterExample}
There exists a slice knot admitting a  genus 1 Seifert surface $F$ on which there is no surgery curve with vanishing signature function. Moreover there are such examples where $F$ is the unique minimal genus one Seifert surface.
\end{theorem}

Therefore, for classical knots, Levine's strategy needs to be re-thought. In a future paper we will attempt to explain the systematic reason behind this failure in such a way that suggests how it might be repaired.

It is instructive to see how Kauffman's Strong Conjecture relates to other famous conjectures in knot theory:
\begin{itemize}

\item [a.] A special case of Kauffman's Strong Conjecture is the conjecture that ``a knot is slice if and only if its untwisted Whitehead double is slice'' \cite[Problem 1.38]{Problems82}. Since our counterexamples are not Whitehead doubles, this specialized conjecture remains open. 

\item [b.] In the topological category all of the conjectures above are false because the untwisted Whitehead double of \textit{any} knot, $J$, is topologically slice by ~\cite{Freedman84}, but has a genus 1 Seifert surface for which both of surgery curves have the knot type  of $J$. However, if knots of Alexander polynomial one are excluded, Kauffman's strong conjecture was previously open in the topological category. The examples in this paper are counterexamples to this conjecture.

\item [c.] In Section~\ref{sec:sliceribbon} we prove that a \textit{stable}  version of Kauffman's Strong Conjecture is related  to the famous Slice-Ribbon Conjecture. 

\end{itemize}

The authors are grateful to Mark Powell for pointing out an error in an earlier version of Section~\ref{sec:uniqueSS}.

\section{Preliminaries}\label{sec:prelim}

\subsection{Infection}

Our counterexamples to Kauffman's conjectures arise from an operation on knots called infection, which is a mild generalization of the satellite construction. In this subsection we recall the relevant definitions.  Consider oriented knots $K_1,\dots, K_n$ and an $n$-component oriented trivial link  $(\eta_1, \dots , \eta_n)$. Let $N(\eta_i)$ be disjoint tubular neighborhoods of the $\eta_i$.  Let $E=S^3-(N(\eta_1)\cup\dots \cup N(\eta_n))$.  Let $N(K_i)$ be a tubular neighborhood of $K$ and $E(K_i)=S^3-N(K_i)$.  For $i=1,\dots, n$ glue to $E$ a copy of $E(K_i)$ along $\bdry (N(\eta_i))$ so that the meridian of $\eta_i$ is identified with the longitude of $K_i$ and the longitude of $\eta_i$ is identified with the meridian of $K_i$.  Call the resulting 3-manifold $M$.  It is easy to see that $M$ is diffeomorphic to $S^3$.

If $R$ is a knot contained in $E$ and $f:E\to M$ is the natural inclusion then a new knot, $R_{\eta_1,\dots, \eta_n}(K_1,\dots, K_n)$, is given by the image of $R$ under $f$.  Since $M\cong S^3$, this yields a knot in $S^3$.  This knot is called the \textbf{result of infection} of $R$ by $K_1, \dots, K_n$ along $\eta_1,\dots, \eta_n$. If $n=1$ this is the same as the satellite knot of $K_1$ with pattern knot $R$ and axis $\eta_1$.

 The $n$-tuple 
 $(\lnk(R,\eta_1),\dots, \lnk(R,\eta_n))$
  is called the \textbf{winding} of $R_{\eta_1,\dots, \eta_n}$. If the winding is $(0,...,0)$ then the $(n+1)$-component link $(R,\eta_1,\eta_2, \dots, \eta_n)$, also denoted   $R_{\eta_1,\dots, \eta_n}$, is called a \textbf{doubling operator}.

In practice, the process of infection ties the strands of $R$ which pass through the disk bounded by $\eta_i$ into the knot $K_i$, as is depicted schematically in Figure~\ref{fig:Example}.
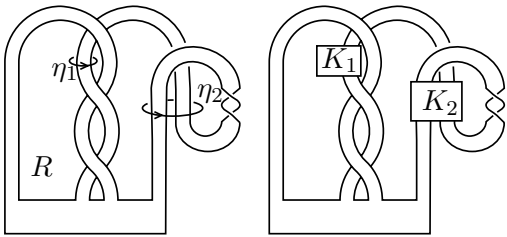
\begin{figure}[htbp]
\setlength{\unitlength}{1pt}
\begin{picture}(180,100)

\put(0,0){\includegraphics[height=1.2in]{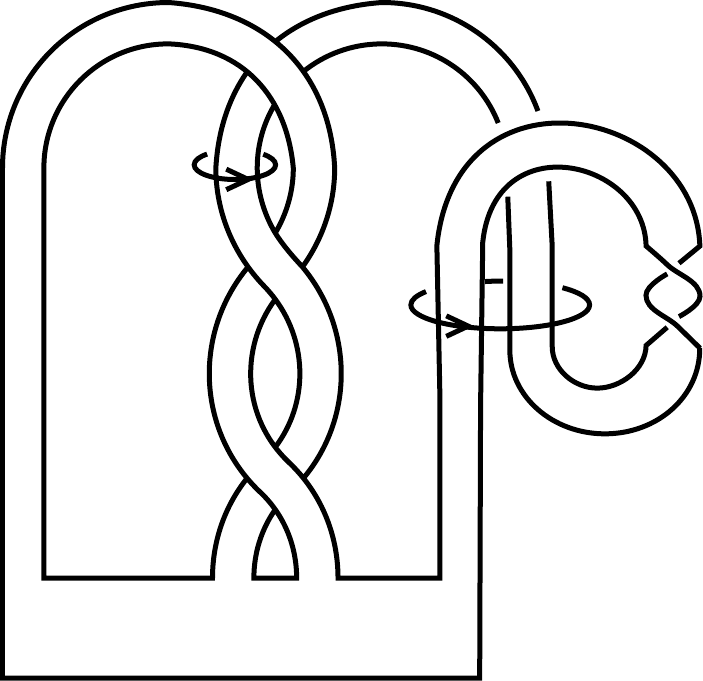}}
\put(18,60){$\eta_1$}
\put(73,52){$\eta_2$}
\put(10,22){$R$}

\put(100,0){\includegraphics[height=1.2in]{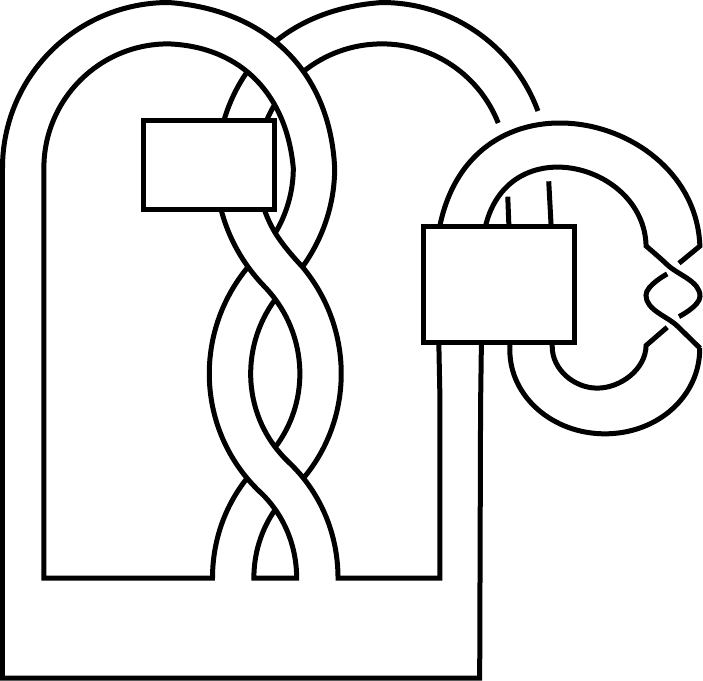}}
\put(120,63){$K_1$}
\put(158,47){$K_2$}

\end{picture}
\caption{\emph{Left:} A doubling operator $R_{\eta_1,\eta_2}$.  \emph{Right:} The result of infection, $R_{\eta_1,\eta_2}(K_1,K_2)$. }\label{fig:Example}
\end{figure}

\subsection{The Arf invariant}

The Arf invariant is a $\Z_2$-valued invariant of knots that is zero for slice knots \cite{Robertello65}.  Levine (and independently Murasugi) showed that it can be computed in terms of the Alexander polynomial. 

\begin{proposition}[\cite{Le66}, see also Theorem 2 of \cite{Murasugi69}]\label{ArfByAlex}
If $K$ is a knot and $\Delta_K(t)$ is the representative of the Alexander polynomial of $K$ satisfying that $\Delta_K(1)=1$ and $\Delta_K(t)=\Delta_K(t^{-1})$ then 
$
\Delta_K(-1) \equiv 1+4\Arf(K) \mod 8.
$
\end{proposition}

The behavior of Alexander polynomials under satellites is well understood.  

\begin{proposition}[Theorem II of \cite{Seifert50}]\label{Seif}
Let $R_\eta$ be a satellite operator with winding number $w = \lnk(R,\eta)$.  For any knot $K$, The Alexander polynomial of the result of infection is given by $\Delta_{R_\eta(K)}(t)=\Delta_R(t)\Delta_K(t^w)$.
\end{proposition}

We can use this to compute the Arf invariant of the result of infection.

\begin{corollary}\label{ArfCor}
For $R_{\eta_1,\dots,\eta_n}$  with winding $(w_1,\dots, w_n)$ and knots $K_1,\dots, K_n$, the Arf invariant of
 the result of infection 
is given by 
\begin{equation}\label{ArfCor goal}
\Arf(R_{\eta_1,\dots,\eta_n}(K_1,\dots,K_n))\equiv \Arf(R)+w_1\Arf(K_1)+\dots +w_n\Arf(K_n)
\end{equation}
\end{corollary}
\begin{proof}
All Alexander polynomials are assumed to be normalized to satisfy the conditions of Proposition~\ref{ArfByAlex}.  We give the proof only in the cases $n=1$ and $n=2$, as these are the only cases needed in this paper.  The general result follows from a straightforward induction.

We begin with the case $n=1$.  

  If $w_1$ is even then $\Delta_{K_1}((-1)^{w_1}) = \Delta_{K_1}(1)=1$.  Proposition~\ref{Seif} implies that $\Delta_{R_{\eta_1}(K_1)}(-1)=\Delta_R(-1)$ so that by Proposition~\ref{ArfByAlex}, $\Arf(R_{\eta_1}(K_1)) = \Arf(R)$.  On the other hand when $w_1$ is even, the right hand side of \eref{ArfCor goal} is also equal to $\Arf(R)$ in $\Z_2$.

If $w_1$ is odd, by Proposition~\ref{Seif} and Proposition~\ref{ArfByAlex}, 
\begin{equation*}
\begin{array}{rl}
\Delta_{R_{\eta_1}(K_1)}(-1) =  \Delta_R(-1)\Delta_{K_1}(-1)&\equiv (1+4\Arf(R))(1+4\Arf(K_1))\\& \equiv 1 + 4(\Arf(R)+\Arf(K_1)) \mod 8
\end{array}
\end{equation*}
On the other hand, also by Proposition~\ref{ArfByAlex}
$$
\Delta_{R_{\eta_1}(K_1)}(-1) \equiv (1+4\Arf(R_{\eta_1}(K_1))) \mod 8
$$
so that $\Arf(R)+\Arf(K_1) = \Arf(R_{\eta_1}(K_1))$, which completes the proof in the case $n=1$.

Consider $R_{\eta_1,\eta_2}$ and knots $K_1$ and $K_2$.  Let $R'$ be the knot $R_{\eta_1}(K_1)$.  Notice that $\eta_2$ can naturally be seen in the complement of $R'$ (See Figure~\ref{fig:InfectOneAtATime}.)  Moreover $\lnk(\eta_2,R')=\lnk(\eta_2,R)$. Thus the results of infection, $R_{\eta_1,\eta_2}(K_1,K_2)$ and $(R_{\eta_1}(K_1))_{\eta_2}(K_2)$, are isotopic. We can now use the result from the case that $n=1$.
  $$
  \begin{array}{rcl}
  \Arf(R_{\eta_1,\eta_2}(K_1,K_2))&=&\Arf(R'_{\eta_2}(K_2))\\
  & =& \Arf(R')+w_2 \Arf(K_2) \\
  & =& \Arf(R_{\eta_1}(K_1))+w_2 \Arf(K_2)\\
  & =& (\Arf(R)+w_1 \Arf(K_1))+w_2 \Arf(K_2) ,
  \end{array}
  $$
 as was claimed.

\end{proof}

\begin{figure}[h!]
\setlength{\unitlength}{1pt}
\begin{picture}(180,100)

\put(0,0){\includegraphics[height=1.2in]{DoublingOperator0.pdf}}
\put(18,60){$\eta_1$}
\put(73,52){$\eta_2$}

\put(100,0){\includegraphics[height=1.2in]{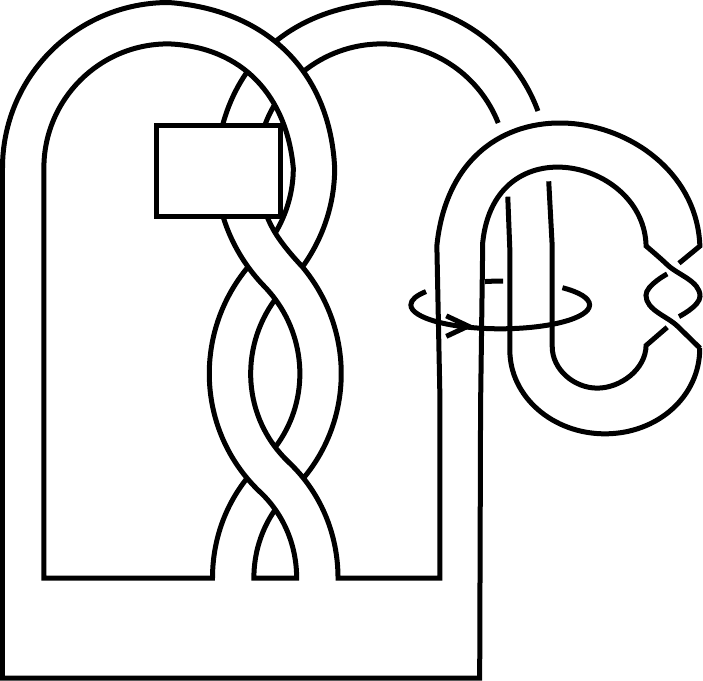}}

\put(173,52){$\eta_2$}
\put(122,62){$K_1$}

\end{picture}
\caption{\emph{Left:} $R_{\eta_1,\eta_2}$.  \emph{Right:} $(R_{\eta_1}(K_1))_{\eta_2}$}\label{fig:InfectOneAtATime}
\end{figure}

 \section{The basic idea: a new method of creating slice knots}

In this section we discuss the main tool of the paper, which is a new method of creating slice knots, and we explain the idea behind why this process can lead to counterexamples to Kauffman's conjectures. 

Suppose $R$ is a knot with a genus one Seifert surface $F$ on which one of the two surgery curves, $d$, is a slice knot. Thus $R$  itself is a slice knot. We first describe a simple procedure to alter the triple $(R,F,d)$ to $(R',F',d')$ where $R'$ is a slice knot with genus one Seifert surface $F'$ on which the concordance type of the surgery curve $d'$ can be assured to be non-zero.  Suppose $\{\eta_1,\eta_2\}$ is a link in the exterior of $F$ with the following properties:
\begin{itemize} 
\item [1.] $\eta_1$ and $\eta_2$ cobound an embedded oriented annulus $A$ that misses $R$ (so $\eta_1$ is isotopic to $\eta_2$ in $S^3-R$);
\item [2.]  $A$ has non-zero algebraic intersection number with $d$;
\item [3.]  $\{\eta_1,\eta_2\}$ is a trivial link in $S^3$.
\end{itemize}
A local picture of the generic situation is shown on the left-most side of Figure~\ref{fig:easyannulus}, where the intersection number is one.

\begin{figure}[h!]
\setlength{\unitlength}{1pt}
\begin{picture}(250,110)
\put(-80,-50){\includegraphics[width=2in]{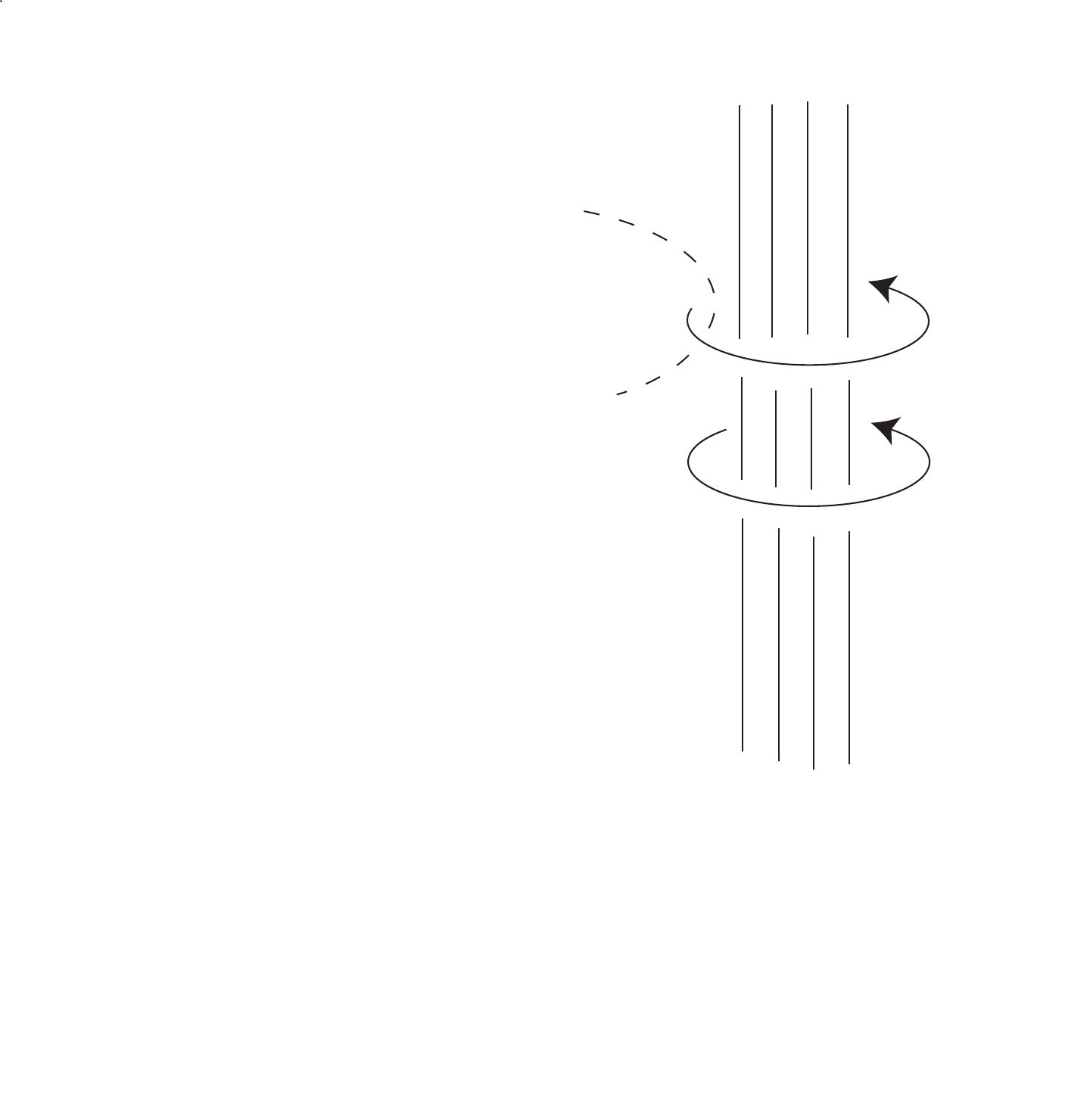}}
\put(0,61){$d$}
\put(40,10){$R$}
\put(52,60){$\eta_1$}
\put(52,35){$\eta_2$}
\put(70,-8){\includegraphics[height=1.44in]{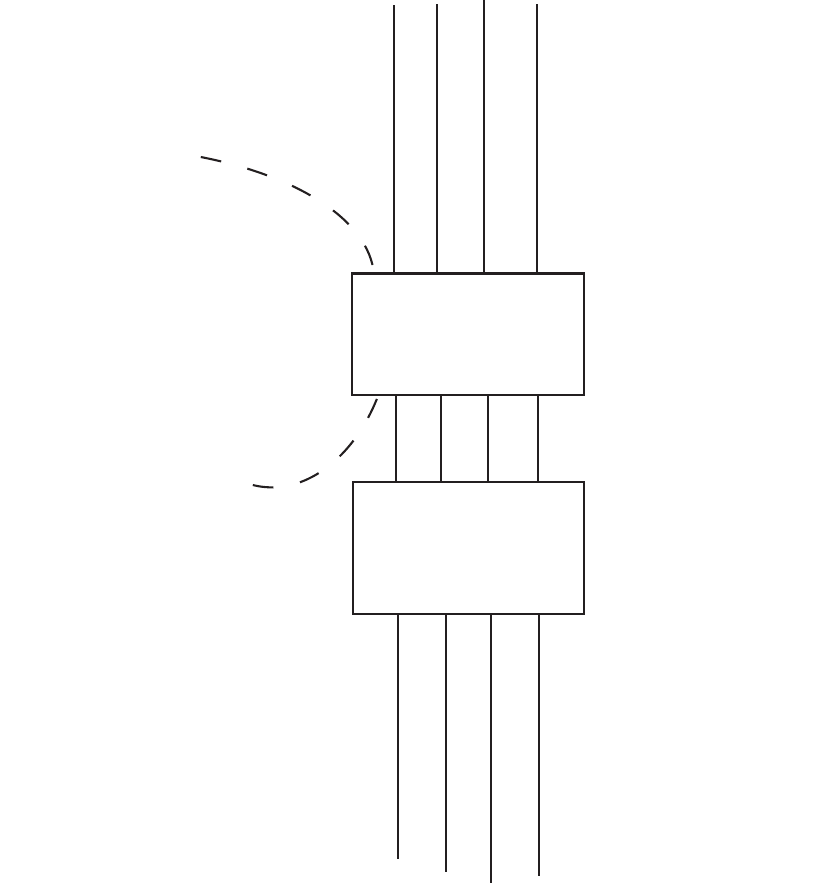}}
\put(120,53){$K$}
\put(100,60){$d'$}
\put(115,28){$-K$}
\put(140,10){$R'$}
\put(140,-10){\includegraphics[height=1.44in]{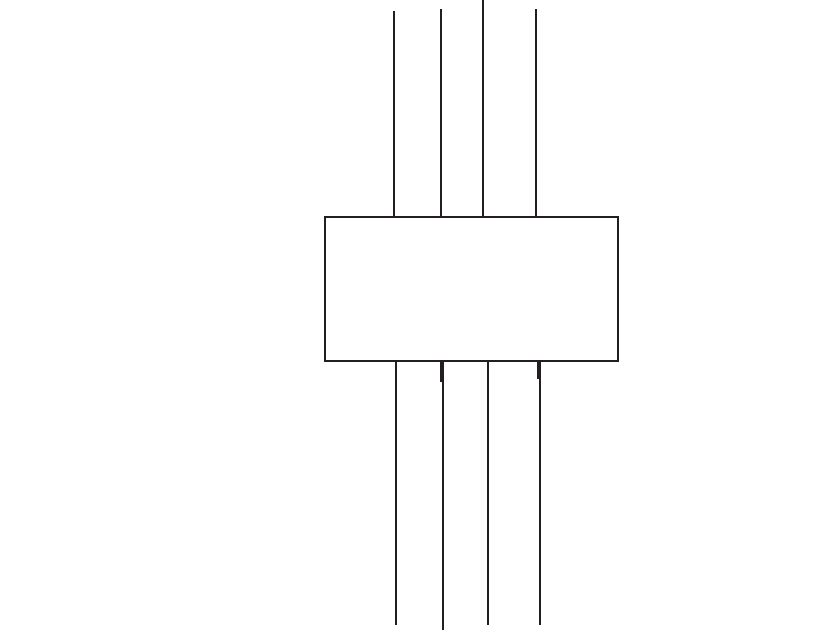}}
\put(198,44){$K\# -K$}
\put(235,10){$R'$}

\end{picture}
\caption{}\label{fig:easyannulus}
\end{figure}

Links satisfying $1$ and $2$ are easy to construct: choose an embedded curve $\eta$ on $F$ that intersects $d$ algebraically non-zero,  let $A\cong \eta\times [-1,1]$ be a small annulus transverse to $F$ and let $\{\eta_1,-\eta_2\}=\partial A$. By construction, this link satisfies $1$ and $2$. If $\eta$ has trivial knot type and $\lnk(\eta,\eta^+)=0$ then it also satisfies $3$.

Returning to the general situation, because of condition $3$, for an arbitrary knot $K$, the knot $R'=R_{\eta_1,\eta_2}(K, -K)$  is defined. Because of condition $1$, the result of the link infection is the same as the result of a single infection on $\eta_1$ by the knot $K\# -K$ as shown in the middle and right-most side of Figure~\ref{fig:easyannulus}. Since $R'$ is a satellite knot of the form $R_{\eta_1}(K\# -K)$ where both $R$ and $K\# -K$ are slice knots ~\cite[Lemma 12.1.2]{Ka3}, it is well-known that $R'$ is a slice knot.  Moreover since $\{\eta_1,\eta_2\}$ misses $F$, the latter survives in $S^3-R$ and is a genus one Seifert surface, $F'$, for $R'$ containing $d'$, the image  of $d$. Moreover, $\lnk(\eta_1,R)=\lnk(\eta_2,R)=0$, so the Seifert form of $R'$ with respect to $F'$ is the same as that of $R$ with respect to $F$. In particular $d'$ is a surgery curve for $R'$. However, by condition $2$, $\lnk(\eta_1,d)\neq \lnk(\eta_2, d)$, so the concordance class of the knot type of $d'$ will inevitably be altered. In the case of Figure~\ref{fig:easyannulus}, the middle figure shows that $d$ is altered to $d'=d\#K$. In general, since $d'$ is the result of infection on $d$ along $\{\eta_1,\eta_2\}$ by $\{K,-K\}$, by formula ~(\ref{ArfCor goal})
$$
\text{Arf}(d')=(w_1-w_2)\text{Arf}(K),
$$
where $w_i=\lnk(\eta_i,d)$.  So if $K$ is a trefoil knot and $w_1-w_2$ is odd then Arf$(d')\neq 0$. There is a similar effect on the signatures of $d$.

Thus we have shown how to alter $(R,F,d)$ to $(R',F',d')$ where $R'$ is a slice knot with genus one Seifert surface $F'$ on which the Arf invariant of the surgery curve $d'$ is non-zero. But recall that there are two possible surgery curves on $F'$. How do we ensure that the second one also has non-zero Arf invariant?  For this it helps to generalize the simple technique above. We do this by showing that we can relax condition $1$ above and allow the annulus to be embedded in the $4$-ball in the exterior of a slice disk for $R$.

The following is the main tool of the paper. 


\begin{theorem}\label{thm:main}
Let $R$ be a slice knot bounding a slice disk $D$ in the 4-ball $B^4$.  Let $(\eta_1,\eta_2)$ be an oriented trivial $2$-component link in $S^3$ that is disjoint from $R$.  Suppose that there is a smooth proper embedding of the annulus $\phi:S^1\times [0,1]\into B^4-D$ with $\phi|_{S^1\times\{0\}}=\eta_1$ and $\phi|_{S^1\times\{1\}}=\eta_2$.  Then, for any knot $K$, the result of infection $R_{\eta_1,\eta_2}(K,-K)$ is smoothly slice.
\end{theorem}

The idea behind Theorem~\ref{thm:main} is simple. If we remove from $B^4$ the neighborhood of an annulus, $N(A)\cong A\times D^2$, and replace it with something with the same homology, namely $E(K)\times [0,1]$, then the result, $\B$, will be a homology $B^4$. It is also easy to check that $\B$ is simply-connected, hence is homeomorphic to $B^4$. Since the disk $D$ was disjoint from $A$, it is a smooth slice disk in $\B$ for the resulting knot on the boundary, which can be seen to be $R_{\eta_1,\eta_2}(K,-K)$. If the smooth $4$-dimensional Poincar\'{e} conjecture were known, we would be done. Lacking that, it takes more work to see that $\B$ is \textit{diffeomorphic} to $B^4$. This is done in the next section by studying a handle structure for $\B$. Both the casual reader and the true expert might want to skip the next section.

\begin{remark} The very elementary idea outlined in the first few paragraphs of this section is sufficient to get counterexamples to what probably \textit{should} have been Kauffman's conjecture. Specifically, recall that if $\Delta$ is a slice disk for a knot $R'$ (with $\Delta_{R'}(t)\neq 1$), and $F'$ is a genus one Seifert surface for $R'$, then there is a surgery curve $d'$, \textit{associated to $\Delta$}, in the sense that $d'$ generates the kernel of the map on rational Alexander modules $\mathcal{A}(S^3-K)\to \mathcal{A}(B^4-\Delta)$, as a $\Q$ vector space. Thus the conjecture probably ought to have been that $d'$, \textit{the surgery curve associated to the slice disk}, is itself a slice knot, or has vanishing Arf invariant. The easy procedure outlined above (where $\eta_1$ and $\eta_2$ arise as push-offs of a single curve on $F$) produces counterexamples to this conjecture. Getting counterexamples to Kauffman's actual conjectures seems to require Theorem~\ref{thm:main}, and slightly modified choices of $\eta_1$ and $\eta_2$.
\end{remark}

\section{Proof of Theorem~\ref{thm:main}}\label{proof}


Let $A$ be the image of $\phi$ and $N(A)$ be a tubular neighborhood of $A$.  Then $N(A)\cap \bdry B^4=N(\eta_1)\cup N(\eta_2)$ is the  union of  disjoint tubular neighborhoods of $\eta_1$ and $\eta_2$.  The map $\phi$ extends to an identification $\overline{\phi}:[0,1]\times S^1\times D^2\cong N(A)$ which restricts to the longitudinal identifications $\{0\}\times S^1\times D^2\cong N(\eta_1)$ and $\{1\}\times S^1\times -D^2\cong N(\eta_2)$. This means that 
$\{0\}\times S^1\times \{1\}$ is identified with longitude of $\eta_1$,
$\{0\}\times \{1\}\times \bdry D^2$ is identified with the meridian of $\eta_1$,
  $\{1\}\times S^1\times \{1\}$ is identified with the 
  longitude of $\eta_2$,
  and $\{1\}\times \{1\}\times \bdry D^2$ is identified with the reverse of the meridian of $\eta_2$.
    Henceforth we assume these identifications.  

Let $E(A)=B^4-N(A)$. Then the boundary of $E(A)$ contains  a copy of $[0,1]\times S^1 \times \partial D^2$. Fix any knot $K$. Note that $(\bdry E(K))\times [0,1]\subseteq \bdry(E(K)\times [0,1])$ is also diffeomorphic to $[0,1]\times S^1\times \partial D^2$, where here again we assume the longitudinal identification $\partial E(K)\cong S^1\times \partial D^2$ where $S^1\times \{1\}$ is a longitude of $K$, $\ell_K$, and $\{1\}\times \partial D^2$ is a meridian of $K$, $\mu_K$.  Let $\B$ be the $4$-manifold obtained by identifying $E(A)$ with $[0,1]\times E(K)$ along these two copies of $[0,1]\times S^1 \times \partial D^2$ via the diffeomorphism, $\psi$, that swaps the second and third factors. Using the above identifications we see that
\begin{itemize}
\item the longitude of $\eta_1$ is identified with the meridian of $K$ in $\{0\}\times E(K)$,
\item the meridian of $\eta_1$ is identified with the longitude of $K$ in $\{0\}\times E(K)$, 
\item the 
longitude of $\eta_2$ is identified with the meridian of $K$ in $\{1\}\times -E(K)$,
\item the
reverse of the
 meridian of $\eta_2$ is identified with the longitude of $K$ in $\{1\}\times -E(K)$.
\end{itemize}
Thus, $\bdry \B$ is given by gluing  $S^3-(N(\eta_1)\cup N(\eta_2))$ to $E(K)$ and $-E(K)$ using the above identifications. Since the longitudes of the $\eta_i$ bound disjoint discs in $S^3-(N(\eta_1)\cup N(\eta_2))$, $\bdry \B$ is diffeomorphic to $S^3$. The knot complement $-E(K)$ is orientation preserving diffeomorphic to $E(-K)$ by a diffeomorphism sending the meridian of $K$ to the meridian of $-K$ and the longitude of $K$ to the reverse of the longitude of $-K$.  By making this substitution it becomes clear that $\bdry \B$ is diffeomorphic to $S^3$ by a diffeomorphism sending $R$ to $R_{\eta_1,\eta_2}(K,-K)$.  Since $D$ was assumed to be disjoint from $A$, it follows that $R_{\eta_1,\eta_2}(K,-K)$ is smoothly slice in $\B$. 

 It remains to show that $\B$ is diffeomorphic to the 4-ball.  
 
 \begin{proposition}
$\B$ is diffeomorphic to the 4-ball.
\end{proposition}
\begin{proof}

Recall that the curves $\eta_1$ and $\eta_2$ bound disjoint disks $D_1$ and $D_2$ in $S^3$.  Instead of cutting out $N(A)$ and gluing back in $[0,1]\times E(K)$,  we can cut out a regular neighborhood of the 2-sphere $S=A\cup D_1\cup D_2$ and glue back in the union of $[0,1]\times E(K)$ with a neighborhood of $D_1\cup D_2$.  Since $S$ is a smoothly embedded 2-sphere, its regular neighborhood, $N(S)$, is a copy of $S^2\times D^2$.  Notice that since $S$ is not properly embedded, intersecting $\bdry B^4$ in $D_1\cup D_2$, $N(S)\cap \bdry B^4$ consists of two 3-balls.

Next, consider $[0,1]\times E(K) \cup N(D_1)\cup N(D_2)$.  Call this 4-manifold $W$.  It is obtained by adding two 2-handles, $N(D_1)$ and $N(D_2)$, to $[0,1]\times E(K)$. 
They are added along meridians $m_a$ and $m_b$ of $K$ sitting in $\{a\}\times \bdry E(K)$ and $\{b\}\times \bdry E(K)$ for some $0<a<b<1$.  The framing is induced by pushing $m_a$ and $m_b$ in the $[0,1]$-direction.  
  The resulting 4-manifold is diffeomorphic to  $S^2\times D^2$ by the following Lemma.


\begin{lemma}\label{trivialize knot}
Let $K$ be any knot.  Consider numbers $0<a<t<b<1$.  Let $m_a$ and $m_b$  be meridians  of $K$ sitting in $\{a\}\times \bdry E(K)$ and $\{b\}\times \bdry E(K)$ for some $0<a<b<1$.  
Let  $\ell$ be the longitude of $K$ in $\{t\}\times \bdry E(K)$ for some $t\in (0,1)$.  Equip $m_a$, $m_b$ and $\ell$ with the framings induced by pushing  them off themselves in the direction of the $[0,1]$-factor.  Let $W$ be given by adding 2-handles to $[0,1]\times E(K)$ along $m_a$ and $m_b$.  Then 
\begin{enumerate}
\item $W$ is diffeomorphic to $S^2\times D^2$.
\item $\ell$ is (isotopic to) the framed boundary of a disk $\Delta'$ in $W$ such that $W-N(\Delta')$ is diffeomorphic to a 4-ball.
\end{enumerate}
\end{lemma}
\begin{proof}

The proof is assisted by recalling  that $[0,1]\times E(K)$ is diffeomorphic to the complement of a slice disk for $K\#-K$, which is seen by noticing that $E(K)\equiv S^3-N(K)$ is the same as $B^3-N(K')$ where $K'$ is a knotted arc whose closure is the knot type of $K$.

The result of adding to $E(K)$ a 3-dimensional 2-handle along the meridian of $K$ is $B^3$.  Thus, the result of adding to $[0,1]\times E(K)$ a 4-dimensional 2-handle along the meridian, $m_a$, of $K$ (in $\{a\}\times \bdry(E(K))$) with framing induced by the $[0,1]$ direction is $[0,1]\times B^3\cong B^4$.  The cocore of the 2-handle is a slice disk for $K\# -K$ whose removal cancels the handle addition and produces $[0,1]\times E(K)$.  It is an interesting exercise to see that the boundary of this slice disk (the belt sphere) is the knot $K\#-K$, but this result is not necessary to our analysis.

Let $h_a$ and $h_b$ be the handles added along $m_a$ and $m_b$ respectively.  The framed curve $m_b$ is isotopic in $[0,1]\times E(K)$ to $m_a$.  Thus, $[0,1]\times E(K)\cup h_a \cong B^4$ via a diffeomorphism sending $m_b$ to the zero framing of a trivial knot, since after adding $h_a$,  $m_b$ bounds an embedded disk in $\partial B^4$. Thus the 2-handle $h_b$ is then added to the zero framing of a trivial knot in $B^4$, producing $S^2\times D^2$.  This completes the proof of the first claim.

\begin{figure}[h!]
\setlength{\unitlength}{1pt}
\begin{picture}(210,55)

\put(0,5){\includegraphics[width=.96in]{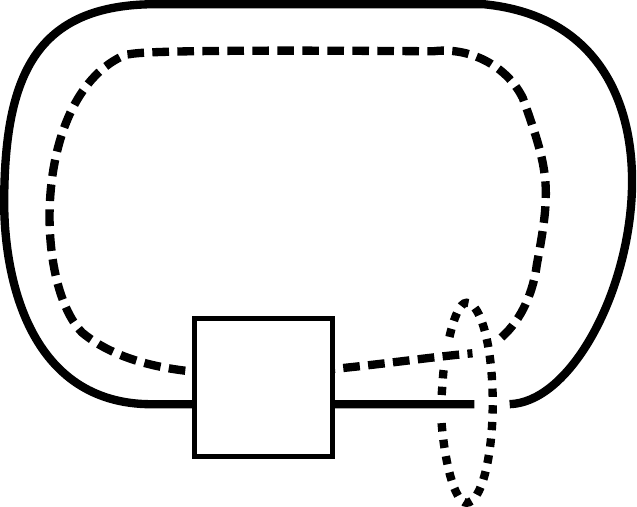}}
\put(24,15){$K$}
\put(50,0){$m'$}
\put(45,44){$\ell'$}

\put(120,5){\includegraphics[width=.96in]{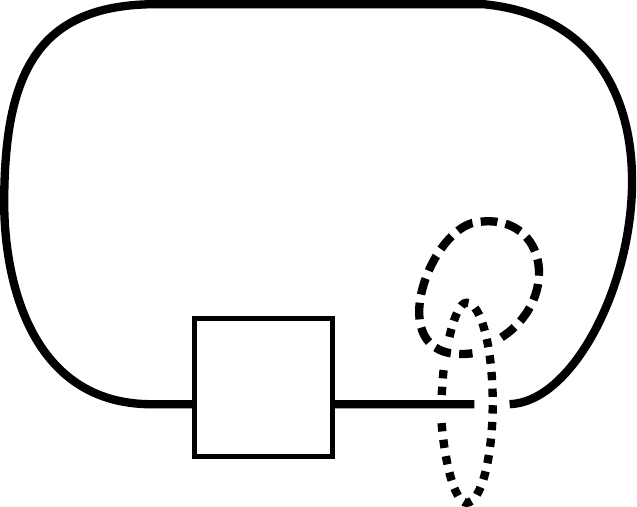}}
\put(144,15){$K$}
\put(170,0){$m'$}
\put(170,38){$\ell'$}
\end{picture}
\caption{Left:  In $\bdry([0,1]\times E(K))$ $\ell$ isotopes to a longitude, $\ell'$ for $K$.  Right:  By sliding $\ell'$ over  $h_1$, $\ell$ is isotoped in $\bdry W$ to the meridian of $\{1\}\times m'$.  }\label{fig:trivializeKnot}
\end{figure}
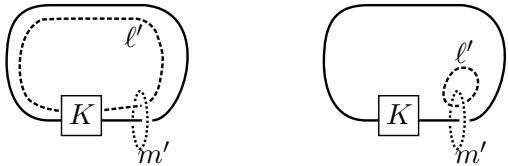

The triple of framed curves $(m_a,m_b,\ell)$ is isotopic to $(\{0\}\times m', \{1\}\times m', \{1\}\times\ell')$ where $m'$ and $\ell'$ are the zero framings of a meridian and preferred longitude of $K$.  The curves $\ell'$ and $m'$ sitting in $\{1\}\times E(K)$ are depicted in Figure~\ref{fig:trivializeKnot}.  By sliding $\{1\}\times \ell'$ over $h_b$ we see that $\ell$ is isotopic in $\bdry W$ to the meridian of $\{1\}\times m$.  This is further isotopic to the belt sphere of the 2-handle $h_b$.  The belt sphere bounds the cocore $\Delta'$ of $h_b$.  Adding a 2-handle and then removing its cocore undo each other:  
$$
\begin{array}{rcl}
W-N(\Delta')&\cong& [0,1]\times E(K)\cup h_a\cup h_b - N(\text{cocore of }h_b)
\\&\cong&[0,1]\times E(K)\cup h_a
\cong 
 [0,1]\times B^3 \cong B^4.
\end{array}
$$


\end{proof}

Thus, $\B$ is given by cutting out $N(S)\cong S^2\times D^2$ and gluing back in $W\cong S^2\times D^2$, a construction reminiscent of the so-called Gluck twist \cite{GluckTwist}.  Of course, $S$ is not properly embedded and so the gluing does not take place on all of $\bdry (N(S))$.  Rather, it occurs on $\bdry N(S)-\bdry B^4$.  Since $S\cap \bdry (B^4) = D_1\cup D_2$,  $\bdry N(S) \cap \bdry B^4$ consists of 3-dimensional neighborhoods of $D_1$ and $D_2$, and $\bdry N(S)-\bdry B^4$ is given by the complement of two 3-balls, $B_1$ and $B_2$ in $\bdry N(S)$.  

We see that the 4-manifold $\B$ is given by cutting out $N(S)$ and gluing back in $W$ using an embedding $\psi_0:(\bdry N(S)-B_1-B_2)\into \bdry W\cong S^1\times S^2$.  It is easy to show that $\psi_0$ induces an isomorphism $\pi_1(\bdry N(S)-B_1-B_2)\to \pi_1(\bdry W)$.  This implies that the 2-sphere $\psi[\bdry B_i]$  bounds a 3-ball in $\bdry W$.  Using this 3-ball we can extend $\psi_0$ over $B_i$ and so to a diffeomorphism $\psi_1:\bdry N(S)\to \bdry W$.  It remains to determine if $\psi_1$ extends to a diffeomorphism $N(S)\to W$.  

Recall that the original gluing map, $\psi$, identifies the meridian of $\eta_1$ with the longitude of $K$ in $\{0\}\times E(K)$.   It similarly identifies $m(A)$, the meridian of $A$ in the interior of $B^4$, with $\ell$, the longitude of $K$ in $\{t\}\times E(K)$ for some $t$ as in the statement of Lemma~\ref{trivialize knot}.  Equip $m(A)$ with the framing induced  by pushing it off itself in the $[0,1]$ direction in $N(A)$.  (Recall $A\cong S^1\times[0,1]$ so $N(A)\cong S^1\times[0,1]\times B^2$.)  Equip $\ell$ with the framing induced by pushing $\ell$ off itself in the $[0,1]$ direction.  Notice that $\psi$, and so $\psi_1$ respect these framings. 

In $N(S)$ this framed meridian of $A$ bounds the disk $\Delta$, such that, for some $q\in S^2$ $(N(S),\Delta)\cong(S^2\times D^2, q\times D^2)$.  Thus, $N(S)-N(\Delta)\cong (S^2-N(q))\times D^2\cong D^2\times D^2$ is a 4-ball.  According to the second result of Lemma~\ref{trivialize knot}, $\ell$ bounds a disk $\Delta'$ in $W$ so that $W-N(\Delta')$ is also a 4-ball.

Since $\psi_1$ identifies the framed boundary of $\Delta$ with the the framed boundary of $\Delta'$, $\psi_1$ extends to a map $\bdry N(S)\cup N(\Delta)\to \bdry W\cup N(\Delta')$.  Since each of $N(S)-N(\Delta)$ and $W-N(\Delta')$ are 4-balls and any diffeomorphism of $\partial B^4$ extends to a diffeomorphism of $B^4$ ~\cite{Cerf}, $\psi_1$ extends to a diffeomorphism $N(S)\to W$. Thus $\B$ is diffeomorphic to the 4-ball.

\end{proof}

Thus, $R_{\eta_1,\eta_2}(K,-K)$ is slice in $\B$ which is diffeomorphic to the 4-ball.
 This completes the proof of Theorem~\ref{thm:main}.


\section{Slice knots with nonslice surgery curves}\label{example}
In this section we 
find counterexamples to Conjectures~\ref{conj:Kauffman} and \ref{conj:signature} by
 finding slice knots 
which have genus 1 Seifert surfaces for which each surgery curve has non-vanishing Arf invariant, and non-zero signatures, respectively. We find that the doubling operator $R_{\eta_1,\eta_2}$  of Figure~\ref{fig:DoublingOperator} 
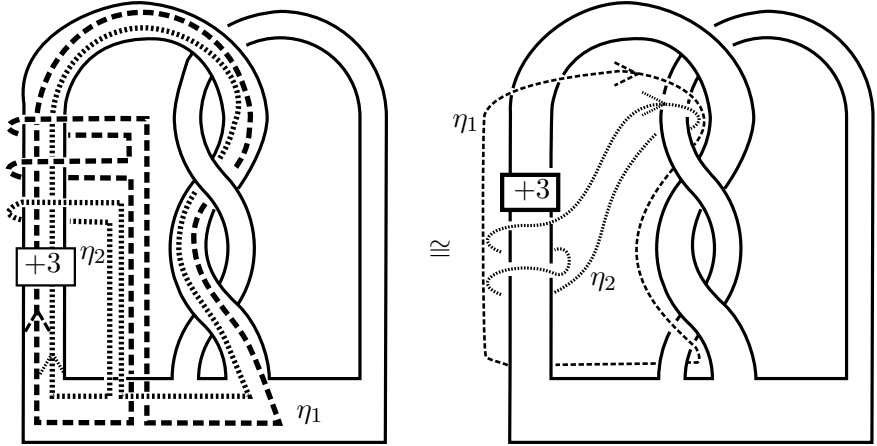
\begin{figure}[h!]
\setlength{\unitlength}{1pt}
\begin{picture}(330,165)
\put(0,0){\includegraphics[height=2.32in]{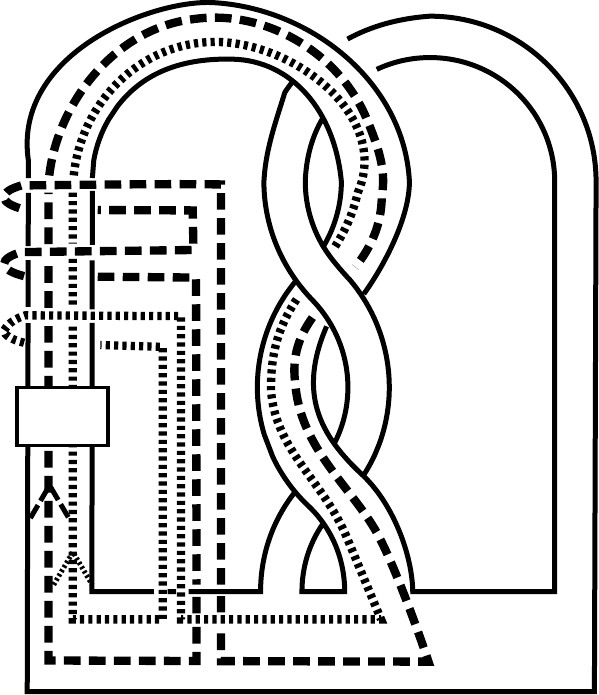}}
\put(110,10){$\eta_1$}
\put(28,70){$\eta_2$}
\put(160,70){$\cong$}
\put(7,65){$+3$}
\put(169,120){$\eta_1$}
\put(221,60){$\eta_2$}
\put(180,0){\includegraphics[height=2.32in]{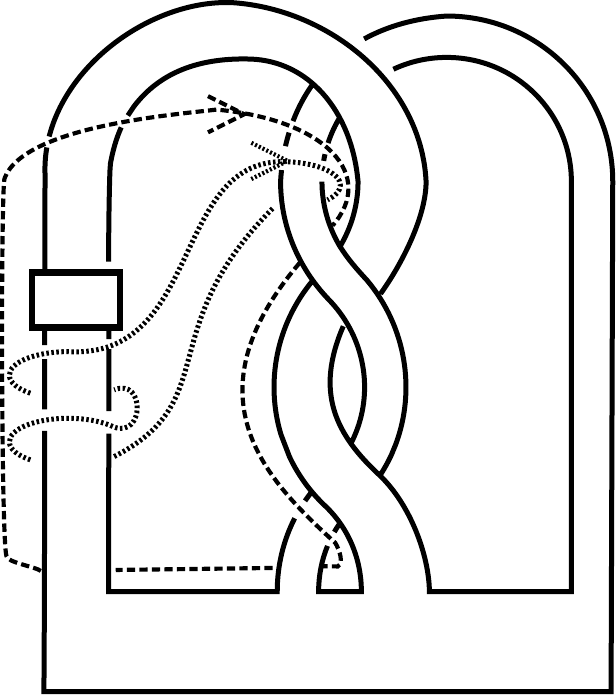}}

\put(192,93){$+3$}

\end{picture}
\caption{The doubling operator operator $R_{\eta_1,\eta_2}$ and an isotopy of it.  Some of the intermediate diagrams are given in Figure~\ref{fig:isotope}}\label{fig:DoublingOperator}
\end{figure}
satisfies the conditions of Theorem~\ref{thm:main}   so that $R_{\eta_1,\eta_2}(K, -K)$ is slice for any knot $K$.  (In the rightmost picture it is clear that $\eta_1$ and $\eta_2$ are unlinked so that  $R_{\eta_1,\eta_2}(K, -K)$ is defined.  The isotopy between the diagrams is depicted in Figure~\ref{fig:isotope}.)

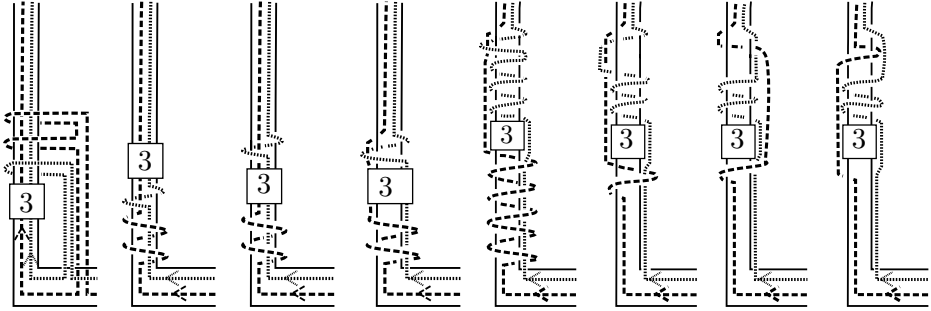
\begin{figure}[h!]
\setlength{\unitlength}{1pt}
\begin{picture}(450,120)
\put(0,0){\includegraphics[height=1.6in]{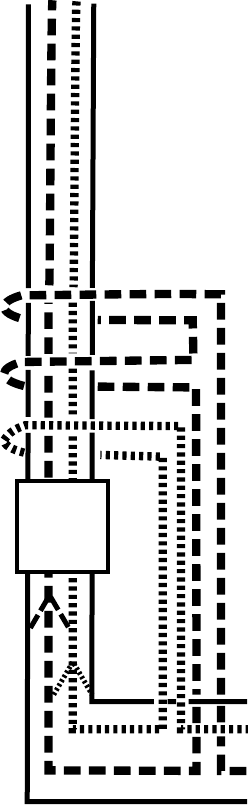}}
\put(6,36){3}

\put(45,0){\includegraphics[height=1.6in]{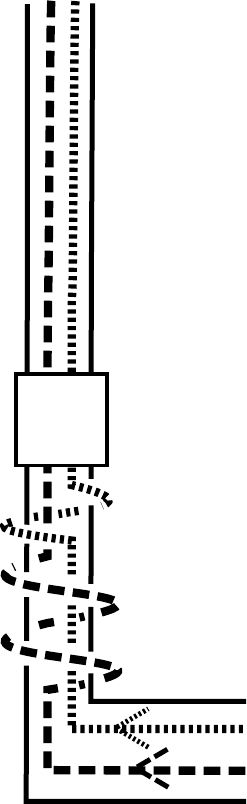}}
\put(51,52){3}

\put(90,0){\includegraphics[height=1.6in]{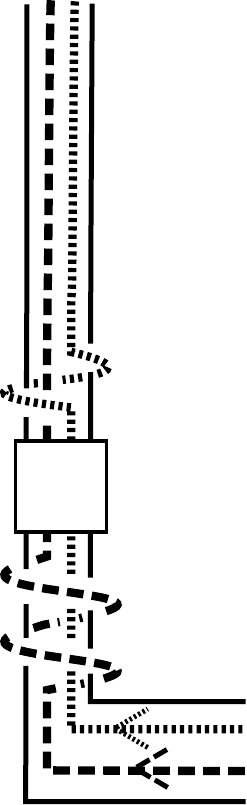}}
\put(96,43){3}

\put(135,0){\includegraphics[height=1.6in]{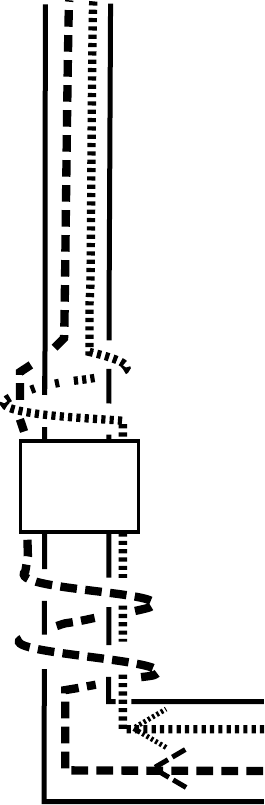}}
\put(142,42){3}

\put(180,0){\includegraphics[height=1.6in]{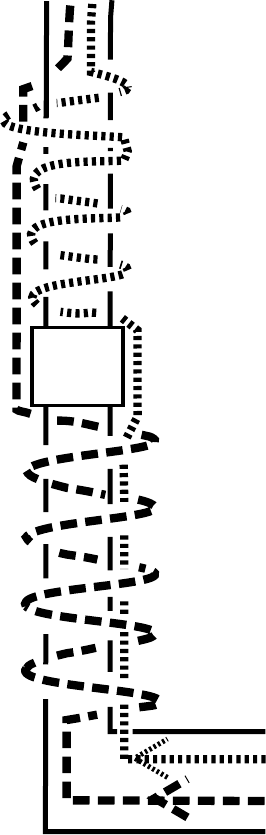}}
\put(188,62){\small3}

\put(225,0){\includegraphics[height=1.6in]{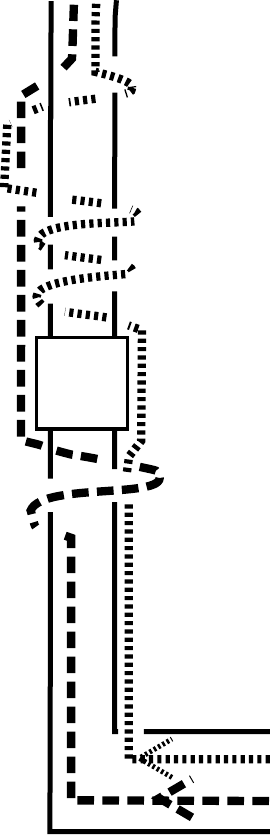}}
\put(233,60){3}

\put(270,0){\includegraphics[height=1.6in]{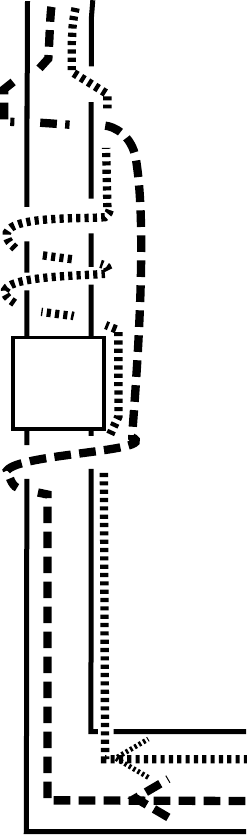}}
\put(275,60){3}

\put(315,0){\includegraphics[height=1.6in]{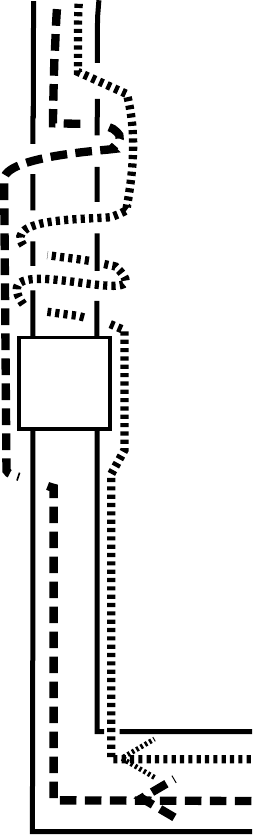}}
\put(321,60){3}
\end{picture}
\caption{Detail of the isotopy of Figure~\ref{fig:DoublingOperator}.}
\label{fig:isotope}
\end{figure}

 We find that if $K$ has nonvanishing Arf invariant, then the surgery curves sitting on the obvious genus 1 Seifert surface for $R_{\eta_1,\eta_2}(K,-K)$ each have nonzero Arf invariant.  We also show that for many choices of $K$ (including the trefoil knot) these surgery curves also have nonvanishing signature function.  These examples prove the first parts of Theorems~\ref{SignatureCounterExample} and ~\ref{counterex}.

\begin{proposition}\label{prop:example}
For the doubling operator $R_{\eta_1, \eta_2}$ of Figure~\ref{fig:DoublingOperator} and any knot $K$, the result of infection $R_{\eta_1,\eta_2}(K,-K)$ is smoothly slice.
\end{proposition}
\begin{proof}

Notice that the existence of the annulus required by Theorem~\ref{thm:main} is equivalent to the claim that $\eta_1$ and $r(\eta_2)$, the reverse of $\eta_2$ together form the oriented boundary of an embedded annulus disjoint from the slice disk for $R$.  

\noindent By adding a band between $\eta_1$ and $r(\eta_2)$ (Figure~\ref{fig:FindAnn1}) and then another from $R$ to itself (Figure~\ref{fig:FindAnn2}) we build a two component cobordism $B$ between $(R,\eta_1,r(\eta_2))$ and the three component unlink $U=(U_1, U_2, U_3)$.  This cobordism consists of two pairs of pants, one with boundary $R\cup U_1 \cup U_2$ the other $\eta_1\cup r(\eta_2) \cup U_3$.  Since $U$ is the unlink it bounds the disjoint union of three disks.  By gluing these disks to $B$ we see an annulus bounded by $\eta_1\cup \eta_2$ disjoint from a disk bounded by $R$.   Thus, Theorem~\ref{thm:main} applies and for any knot $K$ the result of infection $R_{\eta_1,\eta_2}(K,-K)$ is smoothly slice.

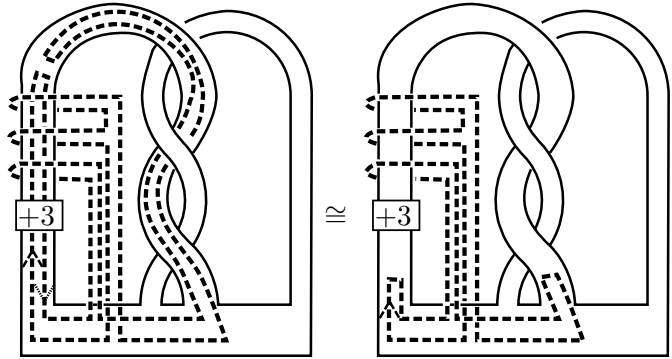
\begin{figure}[h!]
\setlength{\unitlength}{1pt}
\begin{picture}(260,150)
\put(0,0){\includegraphics[width=1.6in]{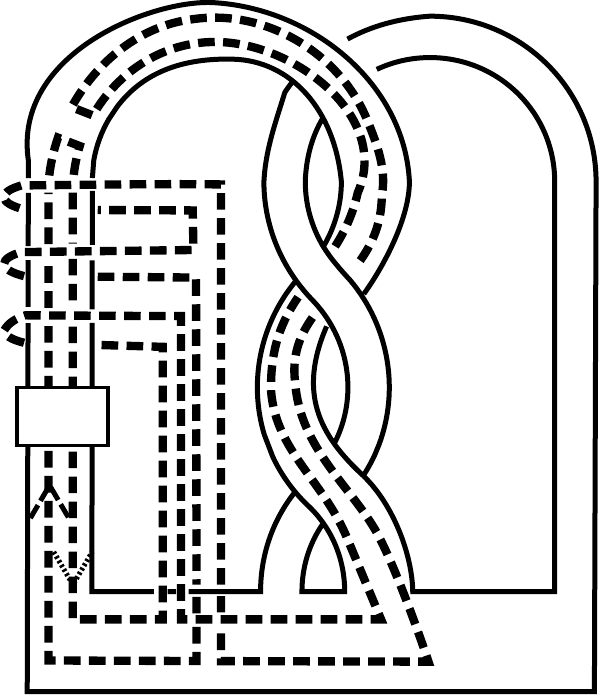}}
\put(4,50){$+3$}
\put(120,50){$\cong$}
\put(135,0){\includegraphics[width=1.6in]{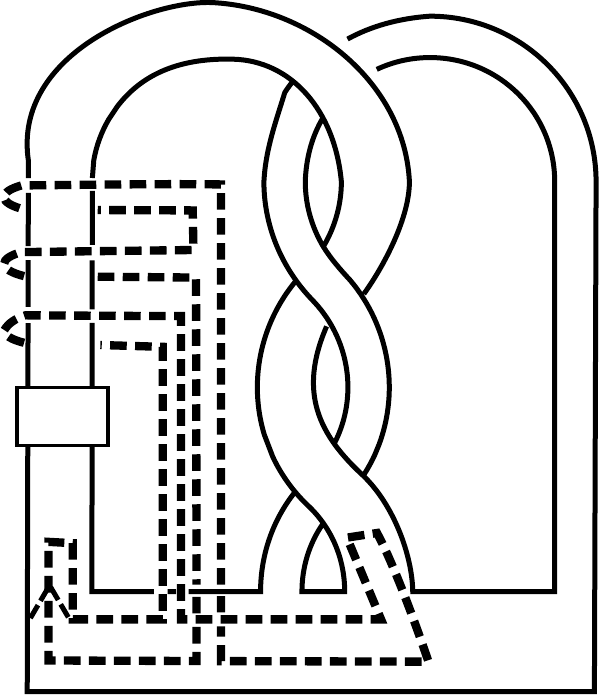}}
\put(139,50){$+3$}
\end{picture}
\caption{ Adding a band between $\eta_1$ and the reverse of $\eta_2$ and an isotopy}
\label{fig:FindAnn1}
\end{figure}

\begin{figure}[h!]
\setlength{\unitlength}{1pt}
\begin{picture}(260,150)
\put(0,0){\includegraphics[width=1.6in]{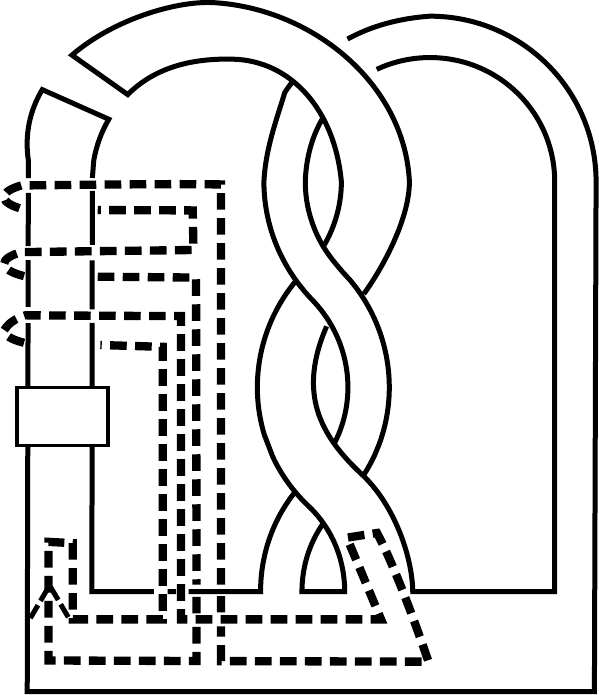}}
\put(4,50){$+3$}
\put(120,50){$\cong$}
\put(135,0){\includegraphics[width=1.6in]{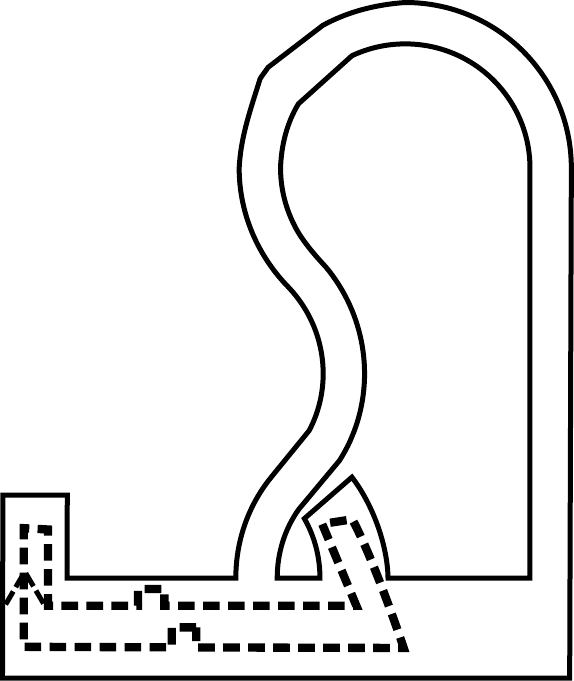}}
\end{picture}
\caption{ Performing a ribbon move for $R$ reveals an unlink}
\label{fig:FindAnn2}
\end{figure}


\end{proof}

For aesthetic reasons, we include one final isotopy of the diagram for $R_{\eta_1,\eta_2}$ on the left side of Figure~\ref{fig:DoublingOperatorBox}.  Since this diagram is not needed for our analysis, we provide no argument for this isotopy.  Then the knot on the right side of  Figure~\ref{fig:DoublingOperatorBox} gives a simple schematic picture of our counterexample to Kauffman's conjecture.

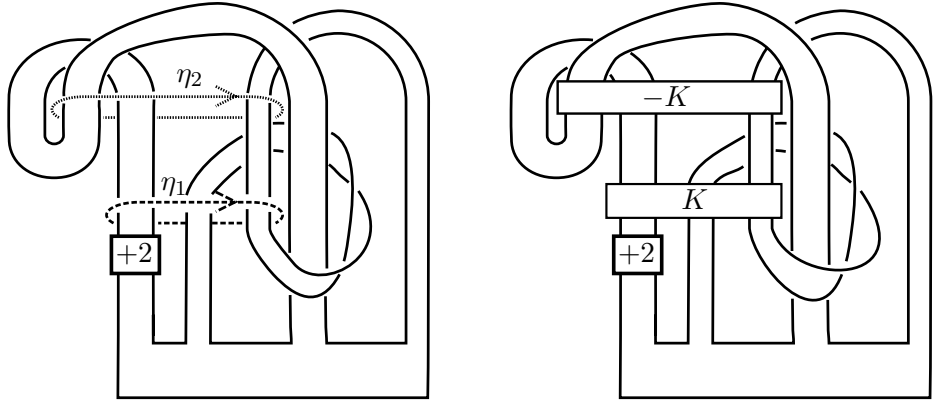
\begin{figure}[h!]
\setlength{\unitlength}{1pt}
\begin{picture}(350,155)
\put(0,0){\includegraphics[height=2.08in]{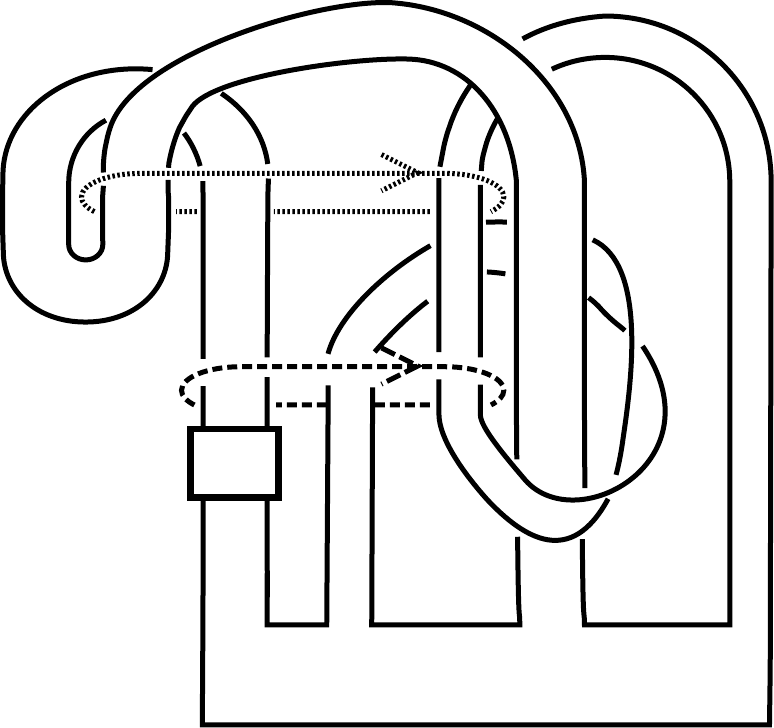}}
\put(64,120){$\eta_2$}
\put(58,79){$\eta_1$}
\put(41,52){$+2$}

\put(190,0){\includegraphics[height=2.08in]{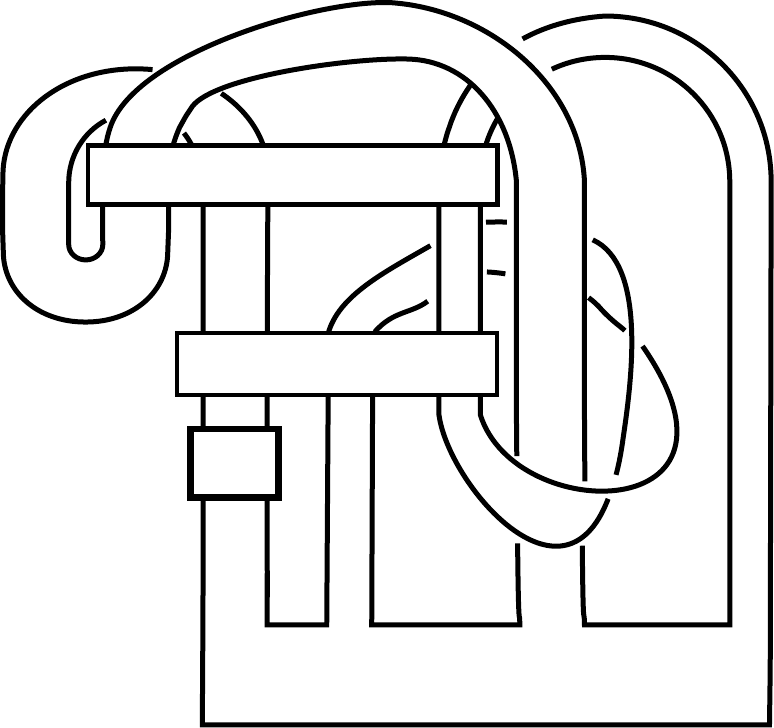}}

\put(240,111){$-K$}
\put(255,72){$K$}
\put(231,52){$+2$}

\end{picture}
\caption{Left:  $R_{\eta_1,\eta_2}$.  Right:  A diagram for $R_{\eta_1,\eta_2}(K,-K)$.  
}\label{fig:DoublingOperatorBox}
\end{figure}

Next, we compute the Arf invariants of the surgery curves on a Seifert surface for the slice knot $R_{\eta_1,\eta_2}(K,-K)$.

\begin{proposition}
For the doubling operator $R_{\eta_1,\eta_2}$ of Figure~\ref{fig:DoublingOperator} and any knot $K$,  $R_{\eta_1,\eta_2}(K, -K)$ bounds a certain genus 1 Seifert surface $F$.  If $\Arf(K)\neq0$ then both surgery curves on $F$ have nontrivial Arf invariant.
\end{proposition}
\begin{proof}
Let $F_0$ be the obvious genus one Seifert surface for $R$ shown in  Figure~\ref{fig:NonDoubleDerivative}. The two surgery curves, $a$ and $b$, on $F_0$ are also shown there. The curves $\eta_1$ and $\eta_2$ are disjoint from $F_0$ as shown in Figure~\ref{fig:DoublingOperator}.    In Figure~\ref{fig:derivatives} these curves are re-drawn together. 

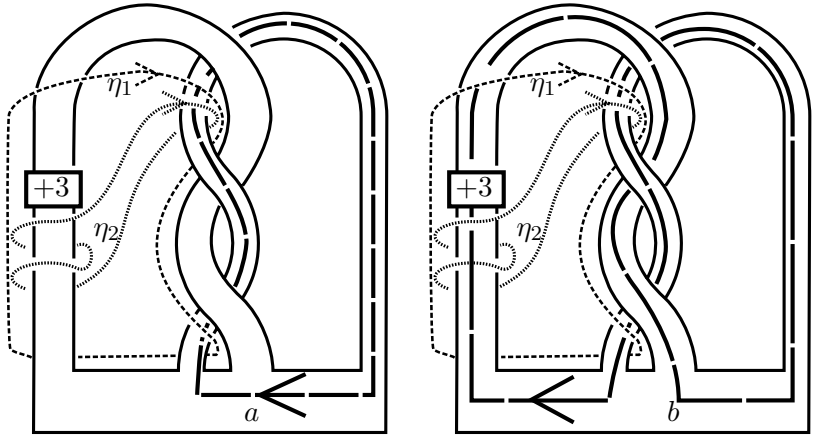
\begin{figure}[h!]
\setlength{\unitlength}{1pt}
\begin{picture}(310,165)
\put(0,0){\includegraphics[width=2in]{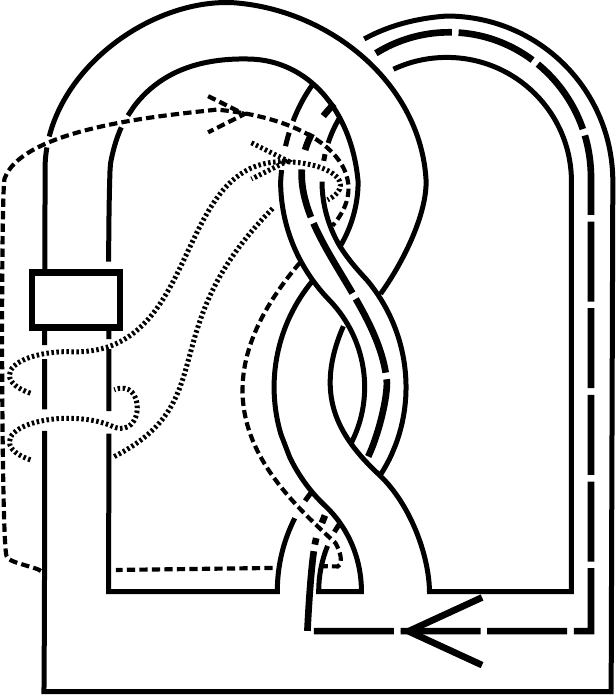}}
\put(160,0){\includegraphics[width=2in]{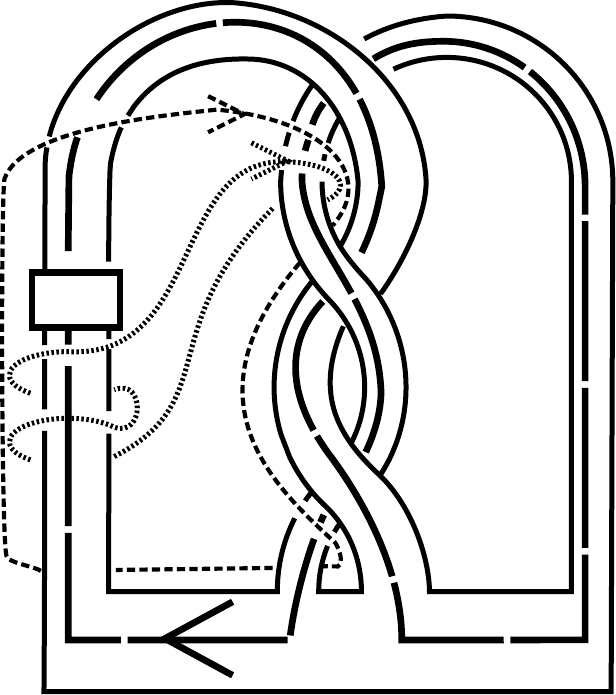}}

\put(38,130){$\eta_1$}
\put(33,75){$\eta_2$}
\put(10,90){$+3$}

\put(90,5){$a$}

\put(198,130){$\eta_1$}
\put(193,75){$\eta_2$}
\put(170,90){$+3$}
\put(250,5){$b$}

\end{picture}
\caption{Surgery curves $\{a,b\}$ on a Seifert surface for $R$ together with the infecting curves
$\{\eta_1,\eta_2\}$}
\label{fig:derivatives}
\end{figure}

Since $\eta_1$ and $\eta_2$ are disjoint from $F_0$, a Seifert surface, $F$, for $R_{\eta_1,\eta_2}(K,-K)$ is given by tying the bands of $F_0$ which pass through the disks bounded by $\eta_1$ and $\eta_2$ into $K$ and $-K$.  Note that this does not introduce any twisting into these bands, so that the Seifert matrix for $R_{\eta_1,\eta_2}(K,-K)$ with respect to the obvious basis of $H_1(F)$ is the same as that for $R$ with respect to $F_0$. Moreover the surgery curves on $F$  have the knot types of $a_{\eta_1,\eta_2}(K,-K)$ and $b_{\eta_1,\eta_2}(K,-K)$.  We can use Corollary~\ref{ArfCor} to compute the Arf invariants of these.  The linking numbers  are given by
\begin{eqnarray}\label{eqn:linkingnos}
\lnk(a,\eta_1)=2,& \lnk(a,\eta_2)=1\\
\lnk(b,\eta_1)=-1,& \lnk(b,\eta_2)=1\nonumber.
\end{eqnarray}
Since $a$ is unknotted and $b$ is the trefoil knot, $\Arf(a)=0$ and $\Arf(b)\neq 0$.   By Corollary~\ref{ArfCor},  
\begin{equation}\label{eq:arf1}
\begin{array}{c}
\Arf(a_{\eta_1,\eta_2}(K,-K)) \equiv \Arf(a)+2\Arf(K)+\Arf(-K) \equiv  \Arf(K)\neq 0,\\ 
\Arf(b_{\eta_1,\eta_2}(K,-K)) \equiv \Arf(b)- \Arf(K)+\Arf(-K) \equiv \Arf(b)\neq 0.
\end{array}
\end{equation}
\end{proof}

Finally, we show that for most of our slice knots $R_{\eta_1,\eta_2}(K,-K)$,  both surgery curves have non-vanishing signature functions. This 
provides a counterexample to the Conjecture~\ref{conj:signature}. Throughout we use the normalized signature function obtained from Levine's signature function by re-defining its value at points of discontinuity to be the average of the values on either side.

As is shown by Litherland in \cite[Theorem 2]{Litherland1979}, this signature invariant behaves very well under satellite operations.  For a knot $K$ and operator $R_\eta$ with $\lnk(R,\eta)=w$ (i.e. a satellite with pattern $R$ and winding number $w$), 
\begin{equation}\label{signatures}
\sigma_{R_\eta(K)}(z) = \sigma_R(z)+\sigma_K(z^w).
\end{equation}
Moreover, $\sigma_{-K}(z) = -\sigma_K(z)$ and $\sigma_K(z^{-1}) = \sigma_K(z)$. We use equations ~\eref{signatures} and ~\eref{eqn:linkingnos} to compute the signature function of the surgery curves of  the Seifert surface for  $R_{\eta_1,\eta_2}(K,-K)$: 
\begin{equation}\label{eq:sig2}
\begin{array}{c}
\sigma_{a_{\eta_1,\eta_2}(K,-K)}(z) =\sigma_{a}(z) + \sigma_K(z^2) - \sigma_K(z) = \sigma_K(z) - \sigma_K(z^2),\\ 
\sigma_{b_{\eta_1,\eta_2}(K,-K)}(z) = \sigma_b(z)+\sigma_K(z) - \sigma_{K}(z^{-1}) = \sigma_b(z)
\end{array}
\end{equation}
Since the knot type of $b$ is the trefoil knot, its signature function is non-zero.  Thus the signature function of $b_{\eta_1,\eta_2}(K,-K)$ is non-zero. Furthermore, if $K$ is any knot for which $\sigma_K(z^2) - \sigma_K(z)$ does not vanish identically (for example the trefoil knot) then the signature function of the surgery curve $a_{\eta_1,\eta_2}(K,-K)$ is not identically zero.  This completes the proof of the first part of  Theorem~\ref{SignatureCounterExample}.



\section{Examples with unique minimal genus Seifert surfaces}\label{sec:uniqueSS}

\begin{theorem}\label{thm:uniquegenusone} There exists a slice knot with a unique genus one Seifert surface, and for this surface both surgery curves have non-trivial Arf invariant and non-trivial signature function. 
\end{theorem}

\begin{proof}  

We will show that the knot $R'=R_{\eta_1,\eta_2}(K,-K)$ with $K$ being the $6_2$ knot, and $R, \eta_1, \eta_2$ being as in Figure~\ref{fig:DoublingOperator} is the desired example.  In the previous section we found that $R'$ is a slice knot and that it has a genus one Seifert surface $F$ on which the two surgery curves have Arf invariants and signatures given by equations ~\ref{eq:arf1} and ~\ref{eq:sig2} upon substituting $K=6_2$.  Since, for the knot $6_2$, both the Arf invariant and the classical signature, $\sigma(-1)$, are non-zero, these equations show that both surgery curves have  non-vanishing signature function and non-zero Arf-invariant.

It remains to show that $F$ is the unique minimal genus Seifert surface for $R'$ (up to isotopy in the exterior of $R'$). First we observe that $R$ itself has a unique minimal genus one Seifert surface (a claim that has appeared previously  in ~\cite[p.2213]{Hor1}).   The argument is that, by a result of M. Hedden, $R$ has the same Knot Floer homology as the $9_{46}$ knot, and this homology is small enough so that, by a result of A. Juhasz, any two minimal genus Seifert surfaces for $R$ are isotopic in the exterior of $R$ ~\cite[Theorem 2.3]{Juh08}.

Now  we proceed by contradiction. Note that $R'$ has non-trivial Alexander polynomial and so has genus one. Suppose that $R'$ admits another genus one Seifert surface that is not isotopic to $F$ in the exterior of $R'$.  Then, by ~\cite{SchTho1988} there exists a genus one surface $F'$ \textit{disjoint from $F$} that is not isotopic to $F$ in the exterior of $R'$. Corollary~\ref{cor:disjointeta} below implies that this $F'$ is  the image of some genus one Seifert surface $\Sigma$ for $R$ (disjoint from $\eta_i$ and $F$). We need to apply Corollary~\ref{cor:disjointeta} in the case that $J=6_2$ and $K=-6_2$, so we use the fact that $6_2$ is a hyperbolic knot to ensure that neither knot exterior admits an incompressible non-boundary parallel annulus. 

\begin{corollary}\label{cor:disjointeta}
Let $R$ be a genus 1 knot with genus 1 Seifert surface $F$.
Let $(\eta_1,\eta_2)$ be a trivial link in the complement of $F$ such that neither of $\eta_1$ and $\eta_2$ bounds a disk in the complement of $F$.  Let $K$ and $J$ be non-trivial knots  which  have no incompressible non-boundary parallel annuli in their exteriors.  
If $F'$ is any genus 1 Seifert surface for $R_{\eta_1,\eta_2}(K,J)$, disjoint from $F$, then there is a genus 1 Seifert surface $\Sigma$ for $R$ which is disjoint from $\eta_1$ and $\eta_2$, such that $F'$ is isotopic in the complement of $F$ to the image of $\Sigma$ in $S^3-R_{\eta_1,\eta_2}(K,J)$.  
\end{corollary}

Two applications of Lemma~\ref{lem:avoideta} below give the proof of Corollary~\ref{cor:disjointeta}.

\begin{lemma}\label{lem:avoideta}
Let $R$ be a genus 1 knot with genus 1 Seifert surface $F$.
Let $\eta$ be an unknot in the complement of $F$ such that $\eta$ does not bound a disk in the complement of $F$.  Let $J$ be a non-trivial knot which  has no incompressible non-boundary parallel annuli in its exteriors.  
If $F'$ is any genus 1 Seifert surface for $R_{\eta}(J)$,disjoint from $F$, then there is a genus 1 Seifert surface $\Sigma$ for $R$ which is disjoint from $\eta$, such that $F'$ is isotopic in the complement of $F$ to the image of $\Sigma$ in $S^3-R_{\eta}(J)$. 
\end{lemma}

\begin{proof}[Proof of Lemma~\ref{lem:avoideta}]
Recall that the exterior of $F_\eta(J)$, $E(F_\eta(J))$, decomposes as the union of $E(J)$ and the exterior of the union of $F$ and $\eta$, $E(F,\eta)$.  In order to prove the lemma it suffices to show that $F'$ can be isotoped in $E(F_\eta(J))$ to lie in $E(F,\eta)$.  Let $T$ be the torus along which they are glued.  Since $\eta$ does not bound a disk in $E(F)$ and $J$ is not the unknot, $T$ is incompressible in $E(F_\eta(J))$.

Isotope $F'$ to be transverse to T.  Then $F'$ decomposes as a union $F'=F_1'\cup F_1'$ identified along some simple closed curves where $F_1'=E(R,\eta)\cap F'$ and $F_1' = E(J)\cap F'$.  
Since $T$ is incompressible in $E(R,\eta)$ and the 3-manifolds $E(J)$ and $E(F,\eta)$ have no incompressible 2-spheres, we may isotope $F'$ until neither $F'_1$ nor $F'_2$ has a disk component.

Since $E(J)$ has no non-boundary parallel annuli, we may isotope $F'$ until $F'_2$ has any annular components.  

Thus, every component of $F_1'$ and $F_2'$ has nonpositive Euler characteristic.  Euler characteristic of surfaces adds, so that $\chi(F'_1)+\chi(F_2')=\chi(F')=-1$, and it follows that either 
$\chi(F_1')=-1$ and $\chi(F_2')=0$ or $\chi(F_1')=0$ and $\chi(F_2')=-1$.  In the latter case $F_1'$ is a union of annuli and we see an annular component bounded by $R$ and $\eta$.  But $R$ (a genus 1 knot) and $\eta$ (an unknot) are not isotopic, so this is impossible.

Thus, it must be that $\chi(F_1')=-1$ and $\chi(F_2')=0$.  But $F_2'$ has neither annular nor disk components.  Every component of $F_2'$ has negative Euler characteristic.  Since $\chi(F_2')=0$, $F_2'$ is empty.  

\end{proof}

This completes the proof of Corollary~\ref{cor:disjointeta}.

Since $R$ has a unique genus one Seifert surface, $F$ and $\Sigma$ are isotopic in the exterior of $R$. Then, by ~\cite[Lemma 3.9]{Manja2009}, $F$ and $\Sigma$ cobound a product region $P$. At least one of the $\eta_i$ must lie in $P$ because otherwise $F$ would be isotopic to $\Sigma$ in the exterior of the $\eta_i$ so $F$ would be isotopic to $F'$ in the exterior of $R'$, contradicting our assumption. If $\eta_i$ lies in $P$ then it would lie in the image of whichever of the two inclusion induced maps factors through $P$:
$$
(i_{\pm})_*: \pi_1(F)\to \pi_1(P)\to \pi_1(S^3-F).
$$
By considering the maps on $H_1$ we will show that the only possibility is that 
$\eta_1$
 lies in $P$ and lies in the image of $(i_-)_*$. 

\begin{figure}[htbp]
\setlength{\unitlength}{1pt}
\begin{picture}(250,120)
\put(70,0){\includegraphics[height=1.6in]{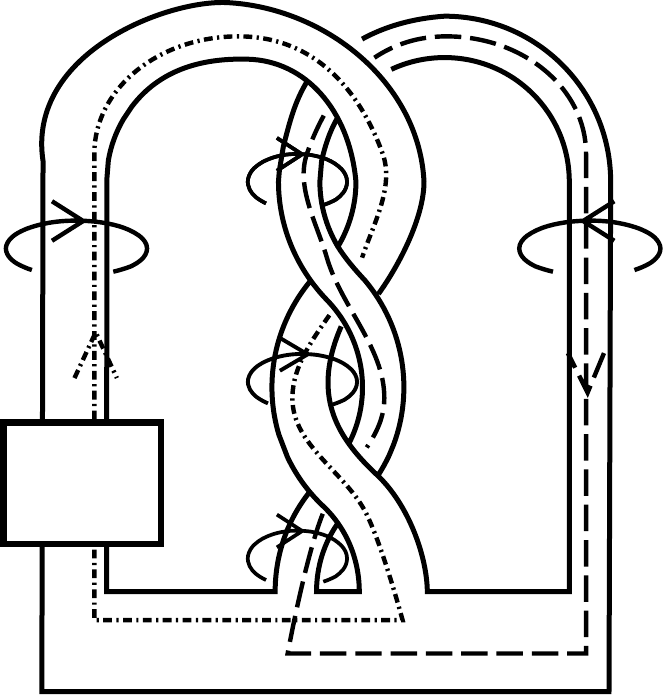}}
\put(75,34){$+3$}
\put(60,74){$\alpha$}
\put(185,73){$\beta$}
\put(105,84){$\delta$}
\put(87,5){$x$}
\put(107,26){$\epsilon$}
\put(100,18){$*$}
\put(180,10){$y$}

\end{picture}
\caption{ Bases for $H_1(F)$ and $H_1(S^3-F)$}\label{fig:Pi1gens}
\end{figure}
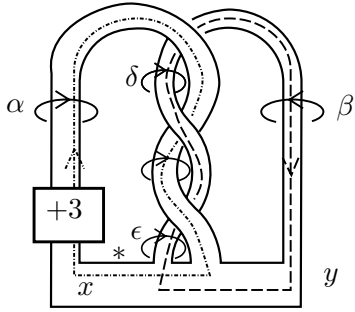
Let the $+$-normal direction to $F$ be upwards out of the plane of the paper, $\{x,y\}$ be the basis for $H_1(F)$ and $\{\alpha,\beta\}$ be the basis for $H_1(S^3-F)$ as in the Figure~\ref{fig:Pi1gens}.  With respect to the bases $\{x,y\}$ and $\{\alpha,\beta\}$  the matrices representing the maps 
$$
(i_{\pm})_*: H_1(F)\to H_1(S^3-F)
$$
are given by the Seifert matrix and its transpose respectively:
$$
i_+=V=\left(\begin{array}{cc}
 3 & 2\\
1 & 0\\
\end{array}\right)
i_-=V^T=\left(\begin{array}{cc}
 3 & 1\\
2 & 0\\
\end{array}\right).
$$
We have 
$\eta_1=\alpha+2\beta$ and $\eta_2=2\alpha+\beta$
 in $H_1(S^3-F)$. The reader can check that $\eta_2$ is not in the column space of $V$ nor $V^T$, so $\eta_2$ is not in the image of either $i_{\pm}$ on $H_1$. Hence $\eta_2$ does not lie in  $P$. The reader can also check that $\eta_1$ is not in the column space of $V$, but is in the column space of $V^T$. Thus the only possibility is that $\eta_1$ lies in $P$ and lies in the image of $(i_-)_*$.
The following $\pi_1$-calculation shows that the latter is not true, completing our proof by contradiction.  The basepoint can be considered to be the point at which $x$ intersects $y$ which can be joined by a small arc \textbf{under the surface} to the point labelled by a $*$ in Figure~\ref{fig:Pi1gens}. 

\begin{lemma}\label{lem:algebra} $\eta_1$ does not lie in the image of $(i_-)_*: \pi_1(F)\to \pi_1(S^3-F)$.
\end{lemma}

\begin{proof}[Proof of Lemma~\ref{lem:algebra}]    The group $\pi_1(F)$ is free on $\{x,y\}$ as shown in that Figure but we will use the basis $\{z,y\}$ where $z=xy$. The group $\pi_1(S^3-F)$ is free on the loops $\{\alpha, \beta\}$ but equally well is free on $\{\alpha,\delta\}$. The basepaths for $\alpha$ and $\delta$ go from the $*$ to the pictured loops in the shortest possible manner. Let $\phi=(i_-)_*$. Then we calculate that:
\begin{eqnarray}\label{eqn:pi1map}
\phi(x)=\alpha^3\g\epsilon=\alpha^4\delta ~\abar \delta ,& \phi(y)=\gbar \alpha\g\\
\phi(z)=\alpha^4\g^2 ,& \eta_1=\phi(x)\abar^2\nonumber.
\end{eqnarray}
Thus $\eta_1$ is in the image of $\phi$ if and only if $\alpha^2$ is in the image of $\phi$. We will show that the element $\alpha^2$ is not in the image of the map of free groups $\phi:F<z,y>\to F<\alpha,\g>$ where $\phi(z)=\alpha^4\g^2$ and $\phi(y)=\gbar \alpha\g$.

\begin{lemma}  For any $k\geq 1$ and $n_i\neq 0$, $1\leq i\leq k$,\\
\indent a. Any element $\phi(...y^{n_{k-1}}z^{n_k})$, when expressed in reduced form, ends in either $\gbar \abar^4$ or $\alpha^4\g^2$. \\
\indent  b. Any element $\phi(...z^{n_{k-1}}y^{n_k})$ , when expressed in reduced form, ends in $\alpha^n\g$ for some $n\neq 0$.
\end{lemma}
\begin{proof} The proof is by induction on $k$. When $k=1$ the result is obvious since $\phi(z^{-1})=\gbar^2\abar^4$ and $\phi(y^{-1})=\gbar\abar\g$. Suppose it is true for $k-1$. 

In case a., we analyze $\phi(...y^{n_{k-1}}z^{n_k})=\phi(...y^{n_{k-1}})\phi(z^{n_k})$. By the inductive hypothesis this ends in  the form $\alpha^n\g (\alpha^4\g^2)^{n_k}$. If $n_k>0$, this is reduced and so ends in $\alpha^4\g^2$ as claimed. If $n_k<0$, this ending has the form
$$
\alpha^n\g \gbar^2\abar^4...\gbar^2\abar^4,
$$
which, after a single reduction, ends in $\gbar\abar^4$ as claimed.

In case b, we analyze $\phi(...z^{n_{k-1}}y^{n_k})=\phi(...z^{n_{k-1}})\phi(y^{n_k})$.  By the inductive hypothesis there are two cases. In the first case this ends in $\gbar \abar^4\gbar \alpha^{n_k}\g$, which is reduced and ends in $\alpha^{n_k}\g$. In the second case this ends in $\alpha^4\g^2\gbar \alpha^{n_k}\g$, which, after a single reduction,  ends in $\alpha^{n_k}\g$.

\end{proof}
Since $\alpha^2$ does not end in one of these forms, it is not in the image of $\phi$. This completes the proof of Lemma~\ref{lem:algebra}.

\end{proof}

This completes the proof of Theorem~\ref{thm:uniquegenusone}.

\end{proof}

\section{Relations with the Slice-Ribbon Conjecture}\label{sec:sliceribbon}

Recall that a knot is a \textbf{ribbon knot} if it is the boundary a disk immersed in $S^3$ whose only singularities are double point arcs.  It is easy to see that any ribbon knot is a slice knot.

\begin{conjecture} [Slice-Ribbon Conjecture]\cite[Problem 1.33]{Problems78}  A knot is a slice knot if and only if it is a ribbon knot.
\end{conjecture}

We say that a Seifert surface $F$ \textbf{stabilizes} to another Seifert surface $F'$, if the latter is obtained from the former by a finite number of trivial one-handle attachments ~\cite[Lemma 5.2.4]{Ka3}. We  call a Seifert surface  a \textbf{Levine surface} if it admits a set of surgery curves that forms a slice link. Then we may state the following.

\begin{conjecture} [Stable Kauffman Conjecture]  If $K$ is a slice knot then any Seifert surface for $K$ stabilizes to a Levine surface.
\end{conjecture}

\noindent Since it is known that any two Seifert surfaces for a knot have a common stabilization, this is equivalent to: If $K$ is a slice knot then  there \textit{exists some} Levine Seifert surface for $K$. Of course, even if true, this is not very useful in practice.


Call $F$ an \textbf{excellent Seifert surface}  if it admits a set of surgery curves that forms a \textit{trivial} link.

\noindent We show:
\begin{proposition}\label{prop:sliceribbon} The Slice-Ribbon Conjecture is equivalent to each of the following:
\begin {itemize}
\item [1.] If $K$ is a slice knot then there exists an excellent Seifert surface for $K$.
\item [2.] If $K$ is a slice knot then any Seifert surface for $K$ stabilizes to an excellent Seifert surface.
\end{itemize}

\end{proposition}

\begin{corollary}\label{cor:sliceribbonimpliesstableK} The Slice-Ribbon Conjecture implies the Stable  Kauffman Conjecture.
\end{corollary}

\begin{proof}  Suppose that the Slice-Ribbon Conjecture holds and that $K$ is a slice knot. Then $K$ is a ribbon knot. Now observe that if $K$ bounds a ribbon disk  $\Delta\looparrowright S^3$  then $K$ admits a natural excellent Seifert surface $F$ obtained by desingularizing the immersed disk $\Delta$ along the set of double arcs in the standard way. Each arc is embedded in the interior of a small sub-disk of $\Delta$. The boundaries of these (disjoint) sub-disks form a set of surgery curves for $F$. Thus $F$ is an excellent Seifert surface. Hence $1.$ follows.

Conversely,  assume $1.$. Then any slice knot $K$ admits some excellent Seifert surface $F$ with associated $g$-component trivial link $T$. Then it is easy to see that $K$ can be obtained from the trivial $2g$-component link formed by taking $2$ anti-parallel copies of $T$ and then  attaching $2g-1$ bands. But the latter is a known description of a ribbon knot. Hence the Slice-Ribbon Conjecture is equivalent to $1.$. 

Clearly $2.$ implies $1.$.  But $1.$ implies $2.$ since, if $F$ is excellent and $F'$ is any another Seifert surface then it is known that $F'$ and $F$ each stabilize to the same (up to isotopy) Seifert surface $F''$. Finally, it is easy to see that any stabilization of an excellent surface is excellent. Thus $F''$ is excellent.  
\end{proof}

In fact, the genus 1 Seifert surface for the slice knot $R_{\eta_1,\eta_2}(K,-K)$ where $(R,\eta_1,\eta_2)$ is as in Figure~\ref{fig:DoublingOperator} \textit{does} stabilize to a Levine surface and the resulting example \textit{is} a ribbon knot. 

Nonetheless, it is possible that  Theorem~\ref{thm:main} could yield candidates for counterexamples to the Slice-Ribbon Conjecture.

\bibliographystyle{plain}
\bibliography{biblio}  

\end{document}